\documentclass[a4paper]{article}

\usepackage[english]{babel}
\usepackage[utf8]{inputenc}
\usepackage[T1]{fontenc}
\usepackage{mathtools}
\usepackage{amsthm}
\usepackage{amsfonts}
\usepackage{amssymb}
\usepackage{pdfpages}
\usepackage{subfig}
\usepackage[figurename=Figure]{caption}
\usepackage[margin=3cm]{geometry}
\usepackage[pdftex,pdfborder={0 0 0},linktoc=all]{hyperref}
\usepackage{tikz}

\theoremstyle{plain}
	\newtheorem{thm}{Theorem}[section]
	\newtheorem{cor}[thm]{Corollary}
	\newtheorem{lem}[thm]{Lemma}
	\newtheorem{prop}[thm]{Proposition}
	\newtheorem*{conj}{Conjecture}

\theoremstyle{definition}
	\newtheorem{dfn}[thm]{Definition}

	\newtheorem{ntns}[thm]{Notations}
	\newtheorem{exa}[thm]{Example}

\theoremstyle{remark}
	\newtheorem{rem}[thm]{Remark}
	\newtheorem{rems}[thm]{Remarks}

\numberwithin{equation}{section}

\newcommand{\C}{\mathbb{C}}
\newcommand{\D}{\mathcal{D}}
\newcommand{\E}{\mathcal{E}}
\newcommand{\I}{\mathcal{I}}
\newcommand{\J}{\mathcal{J}}
\newcommand{\K}{\mathcal{K}}
\newcommand{\N}{\mathbb{N}}
\newcommand{\R}{\mathbb{R}}
\newcommand{\X}{\mathcal{X}}

\newcommand{\dx}{\dmesure\!}
\newcommand{\debar}{\overline{\partial}}
\newcommand{\esp}[2][]{\mathbb{E}_{#1}\!\left[ #2 \right]}
\newcommand{\espcond}[3][]{\mathbb{E}_{#1}\!\left[ #2\hspace{-1mm} \mvert \! #3 \right]}
\newcommand{\odet}[1]{\norm{\det\! ^\perp\!\left(#1\right)}}
\newcommand{\mvert}{\mathrel{}\middle|\mathrel{}}
\newcommand{\norm}[1]{\left\lvert #1 \right\rvert}
\newcommand{\Norm}[1]{\left\lVert #1 \right\rVert}
\newcommand{\pa}{\mathcal{P}\!}
\newcommand{\pp}{\mathcal{P\!P}\!}
\newcommand{\prsc}[2]{\left\langle #1\,, #2 \right\rangle}
\newcommand{\rmes}[1]{\norm{\dmesure\! V_{#1}}}
\newcommand{\var}[1]{\Var\!\left( #1 \right)}
\newcommand{\vol}[1]{\Vol\!\left(#1\right)}

\renewcommand{\L}{\mathcal{L}}
\renewcommand{\P}{\mathbb{P}}
\renewcommand{\S}{\mathbb{S}}
\renewcommand{\H}{H^0(\X,\E \otimes \L^d)}

\renewcommand{\bar}{\overline}
\renewcommand{\epsilon}{\varepsilon}
\renewcommand{\geq}{\geqslant}
\renewcommand{\leq}{\leqslant}
\renewcommand{\tilde}{\widetilde}

\DeclareMathOperator{\card}{Card}
\DeclareMathOperator{\dmesure}{d}
\DeclareMathOperator{\ev}{ev}
\DeclareMathOperator{\Id}{Id}
\DeclareMathOperator{\Var}{Var}
\DeclareMathOperator{\Vol}{Vol}

\DeclarePairedDelimiter{\floor}{\lfloor}{\rfloor}

\usetikzlibrary{graphs}
\usetikzlibrary{arrows}

\author{Michele Ancona\,\thanks{Michele Ancona, Tel Aviv University, School of Mathematical Sciences; e-mail: \url{michi.ancona@gmail.com}. Michele Ancona is supported by the Israeli Science Foundation through the ISF Grants 382/15 and 501/18.} \and Thomas Letendre\,\thanks{Thomas Letendre, Université Paris-Saclay, CNRS, Laboratoire de Mathématiques d’Orsay, 91405 Orsay, France; e-mail: \url{letendre@math.cnrs.fr}. Thomas Letendre is supported by the French National Research Agency through the ANR grants UniRaNDom (ANR-17-CE40-0008) and SpInQS (ANR-17-CE40-0011).}}
\date{\today}
\title{Roots of Kostlan polynomials: moments, strong Law of Large Numbers and Central Limit Theorem}

\begin{document}

\maketitle

\begin{abstract}
We study the number of real roots of a Kostlan random polynomial of degree $d$ in one variable. More generally, we are interested in the distribution of the counting measure of the set of real roots of such polynomials. We compute the asymptotics of the central moments of any order of these random variables, in the large degree limit. As a consequence, we prove that these quantities satisfy a strong Law of Large Numbers and a Central Limit Theorem. In particular, the real roots of a Kostlan polynomial almost surely equidistribute as the degree diverges. Moreover, the fluctuations of the counting measure of this random set around its mean converge in distribution to the Standard Gaussian White Noise. More generally, our results hold for the real zeros of a random real section of a line bundle of degree $d$ over a real projective curve in the complex Fubini--Study model.
\end{abstract}

\begin{description}
\item[Keywords:] Kostlan polynomials, Complex Fubini--Study model, Kac--Rice formula, Law of Large Numbers, Central Limit Theorem, Method of moments.
\item[Mathematics Subject Classification 2010:] 14P25, 32L05, 60F05, 60F15, 60G15, 60G57.
\end{description}


\tableofcontents


\section{Introduction}
\label{sec introduction}

\paragraph{Kostlan polynomials.}

A real \emph{Kostlan polynomial} of degree $d$ is a univariate random polynomial of the form $\sum_{k=0}^d a_k \sqrt{\binom{d}{k}}X^k$, where the coefficients $(a_k)_{0 \leq k \leq d}$ are independent $\mathcal{N}(0,1)$ random variables. Here and in the following we use the standard notation $\mathcal{N}(m,\sigma^2)$ for the real Gaussian distribution of mean $m$ and variance $\sigma^2$. These random polynomials are also known as \emph{elliptic polynomials} in the literature (see~\cite{ST2004} for example). The roots of such a polynomial form a random subset of $\R$ that we denote by $Z_d$. Kostlan proved (cf.~\cite{Kos1993}) that for all $d \in \N$, the average number of roots of this random polynomial is $\esp{\card(Z_d)} = d^\frac{1}{2}$, where $\card(Z_d)$ is the cardinality of $Z_d$. It was later proved by Dalmao (see~\cite{Dal2015}) that $\var{\card(Z_d)} \sim \sigma^2 d^\frac{1}{2}$ as $d \to +\infty$, where $\sigma$ is some explicit positive constant. Dalmao also proved that $\card(Z_d)$ satisfies a Central Limit Theorem as $d \to +\infty$.

In this paper, we study the higher moments of $\card(Z_d)$ in the large degree limit. Let $p \in \N$, we denote by $m_p(\card(Z_d))$ the $p$-th central moment of $\card(Z_d)$. A consequence of our main result (Theorem~\ref{thm main}) is that, as $d \to +\infty$, we have:
\begin{equation*}
m_p(\card(Z_d)) = \mu_p \sigma^p d^\frac{p}{4} + o(d^\frac{p}{4}),
\end{equation*}
where $\sigma$ is the constant appearing in Dalmao's variance estimate, and $(\mu_p)_{p \in \N}$ is the sequence of moments of the standard real Gaussian distribution $\mathcal{N}(0,1)$. This results allows us to prove a strong Law of Large Numbers: $d^{-\frac{1}{2}}\card(Z_d) \to 1$ almost surely. We also prove that $d^{-\frac{1}{2}}\card(Z_d)$ concentrates around~$1$ in probability, faster than any negative power of~$d$. Finally, we recover Dalmao's Central Limit Theorem by the method of moments. The original proof used the Wiener--Ito expansion of the number of roots. In fact, we improve this result by proving a Central Limit Theorem for the counting measure of $Z_d$ (see Theorem~\ref{thm central limit theorem} below).

Equivalently, one can define $Z_d$ as the set of zeros, on the real projective line $\R P^1$, of the homogeneous Kostlan polynomial $\sum_{k=0}^d a_k \sqrt{\binom{d}{k}} X^kY^{d-k}$, where $(a_k)_{0 \leq k \leq d}$ are independent $\mathcal{N}(0,1)$ variables. In this setting, an homogeneous Kostlan polynomial of degree $d$ is a standard Gaussian vector in the space of real global holomorphic sections of the line bundle $\mathcal{O}(d) \to \C P^1$, equipped with its natural $L^2$-inner product (see Section~\ref{subsec geometric setting}). Recall that $\mathcal{O}(d) = \mathcal{O}(1)^{\otimes d}$, where $\mathcal{O}(1)$ is the hyperplane line bundle over $\C P^1$. Then, $Z_d$ is the real zero set of this random section. In this paper, we study more generally the real zeros of random real sections of positive Hermitian line bundles over real algebraic curves.

\paragraph{Framework and background.}

Let us introduce our general framework. More details are given in Section~\ref{subsec geometric setting} below. Let $\X$ be a smooth complex projective manifold of dimension $1$, that is a smooth compact Riemann surface. Let $\E$ and $\L$ be holomorphic line bundles over~$\X$. We assume that $\X$, $\E$ and $\L$ are endowed with compatible real structures and that the real locus of $\X$ is not empty. We denote by $M$ this real locus, which is then a smooth closed (i.e.~compact without boundary) submanifold of $\X$ of real dimension~$1$.

Let $h_\E$ and $h_\L$ denote Hermitian metrics on $\E$ and $\L$ respectively, that are compatible with the real structures. We assume that $(\L,h_\L)$ has positive curvature $\omega$, so that $\L$ is ample and $\omega$ is a Kähler form on $\X$. The form $\omega$ induces a Riemannian metric $g$ on $\X$, hence on $M$. Let us denote by $\rmes{M}$ the arc-length measure on $M$ defined by $g$.

For all $d \in \N$, we denote by $\R \H$ the space of global real holomorphic sections of $\E \otimes \L^d \to \X$. Let $s \in \R \H$, we denote by $Z_s=s^{-1}(0)\cap M$ the real zero set of $s$. Since $s$ is holomorphic, if $s \neq 0$ its zeros are isolated and $Z_s$ is finite. In this case, we denote by $\nu_s$ the counting measure of $Z_s$, that is $\nu_s = \sum_{x \in Z_s} \delta_x$, where $\delta_x$ stands for the unit Dirac mass at $x$. For any $\phi \in \mathcal{C}^0(M)$, we denote by $\prsc{\nu_s}{\phi} = \sum_{x \in Z_s} \phi(x)$. Quantities of the form $\prsc{\nu_s}{\phi}$ are called the \emph{linear statistics} of $\nu_s$. Note that $\prsc{\nu_s}{\mathbf{1}} = \card(Z_s)$, where $\mathbf{1}$ is the constant unit function.

For any $d \in \N$, the space $\R \H$ is finite-dimensional, and its dimension can be computed by the Riemann--Roch Theorem. Moreover, the measure $\rmes{M}$ and the metrics $h_\E$ and $h_\L$ induce a Euclidean $L^2$-inner product on this space (see Equation~\eqref{eq def L2 inner product}). Let $s_d$ be a standard Gaussian vector in $\R \H$, see Section~\ref{subsec zeros of random real sections}. Then, $\nu_{s_d}$ is an almost surely well-defined random Radon measure on $M$. We denote by $Z_d = Z_{s_d}$ and by $\nu_d = \nu_{s_d}$ for simplicity. In this setting, the linear statistics of $\nu_d$ were studied in~\cite{Anc2021,GW2016,Let2016,LP2019}, among others. In particular, the exact asymptotics of their expected values and their variances are known.

\begin{thm}[Gayet--Welschinger]
\label{thm expectation}
For every $d \in \N$, let $s_d$ be a standard Gaussian vector in $\R \H$. Then the following holds as $d \to +\infty$:
\begin{equation*}
\forall \phi \in \mathcal{C}^0(M), \qquad \esp{\prsc{\nu_d}{\phi}} = d^\frac{1}{2} \left(\frac{1}{\pi}\int_M \phi \rmes{M}\right) + \Norm{\phi}_\infty O(d^{-\frac{1}{2}}),
\end{equation*}
where the error term $O(d^{-\frac{1}{2}})$ does not depend on $\phi$. That is: $d^{-\frac{1}{2}}\esp{\nu_d} = \frac{1}{\pi}\rmes{M} + O(d^{-1})$.
\end{thm}

Theorem~\ref{thm expectation} is~\cite[Theorem~1.2]{GW2016} with $n=1$ and $i=0$. See also~\cite[Theorem~1.3]{Let2016} when $\E$ is not trivial. In~\cite{Let2016}, the case the linear statistics is discussed in Section~5.3.

We use the following notation for the central moments of $\nu_d$.

\begin{dfn}[Central moments of $\nu_d$]
\label{def m p nu d}
Let $d \in \N$ and let $s_d$ be a standard Gaussian vector in $\R \H$. Let $\nu_d$ denote the counting measure of the real zero set of $s_d$. For all $p \in \N^*$, for all $\phi_1,\dots,\phi_p \in \mathcal{C}^0(M)$ we denote:
\begin{equation*}
m_p(\nu_d)(\phi_1,\dots,\phi_p) = \esp{\prod_{i=1}^p \left(\prsc{\nu_d}{\phi_i}-\rule{0pt}{10pt}\esp{\prsc{\nu_d}{\phi_i}}\right)}.
\end{equation*}
For all $\phi \in \mathcal{C}^0(M)$, we denote by $m_p(\prsc{\nu_d}{\phi}) = m_p(\nu_d)(\phi,\dots,\phi)$ the $p$-th central moment of $\prsc{\nu_d}{\phi}$. As above, we denote the $p$-th central moment of $\card(Z_d)$ by $m_p(\card(Z_d)) = m_p(\nu_d)(\mathbf{1},\dots,\mathbf{1})$.
\end{dfn}

Of course, $m_p(\nu_d)$ is only interesting for $p \geq 2$. For $p=2$, the bilinear form $m_2(\nu_d)$ encodes the covariance structure of the linear statistics of $\nu_d$. In particular, $m_2(\nu_d)(\phi,\phi) = \var{\prsc{\nu_d}{\phi}}$ for all $\phi \in \mathcal{C}^0(M)$. The large degree asymptotics of $m_2(\nu_d)$ has been studied in~\cite[Theorem~1.6]{LP2019}.

\begin{thm}[Letendre--Puchol]
\label{thm variance}
For all $d \in \N$, let $s_d \in \R\H$ be a standard Gaussian vector. There exists $\sigma >0$ such that, for all $\phi_1$ and $\phi_2 \in \mathcal{C}^0(M)$, the following holds as $d \to +\infty$:
\begin{equation*}
m_2(\nu_d)(\phi_1,\phi_2) =d^\frac{1}{2} \sigma^2\int_M \phi_1 \phi_2 \rmes{M} + o(d^{\frac{1}{2}}).
\end{equation*}
\end{thm}

\begin{rems}
\label{rems thm variance}
\begin{itemize}
\item In fact, $\sigma = \left(\frac{1+\I_{1,1}}{\pi}\right)^\frac{1}{2}$, where $\I_{1,1}$ is the constant appearing in~\cite[Theorem~1.6]{LP2019}. Since~$\sigma$ is universal, it is the same as the one appearing in Dalmao's variance estimate~\cite[Theorem~1.1]{Dal2015}. An integral expression of $\sigma$ is given by~\cite[Proposition~3.1]{Dal2015}. The positivity of $\sigma$ is non-trivial and is given by~\cite[Corollary~1.2]{Dal2015}. See also~\cite[Theorem~1.8]{LP2019}.
\item One can understand how the error term $o(d^{\frac{1}{2}})$ in Theorem~\ref{thm variance} depends on $(\phi_1,\phi_2)$ (see~\cite[Theorem~1.6]{LP2019}).
\end{itemize}
\end{rems}

In~\cite{Anc2021}, Ancona derived a two terms asymptotic expansion of the non-central moments of the linear statistics of $\nu_d$. As a consequence, he proved the following (cf.~\cite[Theorem~0.5]{Anc2021}). 

\begin{thm}[Ancona]
\label{thm non-central moments}
For all $d \in \N$, let $s_d \in \R\H$ be a standard Gaussian vector. Let $p \geq 3$, for all $\phi_1,\dots,\phi_p \in \mathcal{C}^0(M)$ we have: $m_p(\nu_d)(\phi_1,\dots,\phi_p)= o(d^\frac{p-1}{2})$.
\end{thm}

\begin{rem}
\label{rem line bundle E}
In Ancona's paper the line bundle $\E$ is trivial. The results of~\cite{Anc2021} rely on precise estimates for the Bergman kernel of $\L^d$. These estimates still hold if we replace $\L^d$ by $\E \otimes \L^d$, where $\E$ is any fixed real Hermitian line bundle (see~\cite[Theorem~4.2.1]{MM2007}). Thus, all the results in~\cite{Anc2021} are still valid for random real sections of $\E \otimes \L^d \to \X$.
\end{rem}

\paragraph{Main results.}

In this paper, we prove a strong Law of Large Numbers (Theorem~\ref{thm law of large numbers}) and a Central Limit Theorem (Theorem~\ref{thm central limit theorem}) for the linear statistics of the random measures $(\nu_d)_{d \in \N}$ defined above. These results are deduced from our main result (Theorem~\ref{thm main}), which gives the precise asymptotics of the central moments $(m_p(\nu_d))_{p \geq 3}$ (cf.~Definition~\ref{def m p nu d}), as $d \to +\infty$.

Recall that all the results of the present paper apply when $Z_d$ is the set of real roots of a Kostlan polynomial of degree $d$. If one considers homogeneous Kostlan polynomials in two variables, then $Z_d \subset M = \R P^1$. Since $\R P^1$ is obtained from the Euclidean unit circle $\S^1$ by identifying antipodal points, it is a circle of length $\pi$. In this case, $\rmes{M}$ is the Lebesgue measure on this circle, normalized so that $\vol{M}=\pi$. If one wants to consider the original Kostlan polynomials in one variable, then $Z_d \subset M = \R$, where $\R$ is seen as a standard affine chart in $\R P^1$. In this case, the measure $\rmes{M}$ admits the density $t \mapsto (1+t^2)^{-1}$ with respect to the Lebesgue of $\R$, and $\vol{M}=\pi$ once again.

\begin{thm}[Strong Law of Large Numbers]
\label{thm law of large numbers}
Let $\X$ be a real projective curve whose real locus~$M$ is non-empty. Let $\E \to \X$ and $\L \to \X$ be real Hermitian line bundles such that $\L$ is positive. Let $(s_d)_{d \geq 1}$ be a sequence of standard Gaussian vectors in $\prod_{d \geq 1}\R \H$. For any $d \geq 1$, let $Z_d$ denote the real zero set of $s_d$ and let $\nu_d$ denote the counting measure of $Z_d$.

Then, almost surely, $d^{-\frac{1}{2}}\nu_d \xrightarrow[d \to +\infty]{} \frac{1}{\pi}\rmes{M}$ in the weak-$*$ sense. That is, almost surely:
\begin{equation*}
\forall \phi \in \mathcal{C}^0(M), \qquad d^{-\frac{1}{2}}\prsc{\nu_d}{\phi} \xrightarrow[d \to +\infty]{} \frac{1}{\pi} \int_M \phi \rmes{M}.
\end{equation*}
In particular, almost surely, $d^{-\frac{1}{2}}\card(Z_d) \xrightarrow[d \to +\infty]{} \frac{1}{\pi}\vol{M}$ and $\card(Z_d)^{-1}\nu_d \xrightarrow[d \to +\infty]{\text{weak}-*} \frac{\rmes{M}}{\Vol(M)}$.
\end{thm}

\begin{dfn}[Standard Gaussian White Noise]
\label{def W}
The \emph{Standard Gaussian White Noise} on $M$ is a random Gaussian generalized function $W \in \mathcal{D}'(M)$, whose distribution is characterized by the following facts, where $\prsc{\cdot}{\cdot}_{(\D',\mathcal{C}^\infty)}$ denotes the usual duality pairing between $\D'(M)$ and $\mathcal{C}^\infty(M)$:
\begin{itemize}
\item for any $\phi \in \mathcal{C}^\infty(M)$, the variable $\prsc{W}{\phi}_{(\D',\mathcal{C}^\infty)}$ is a real centered Gaussian;
\item for all $\phi_1,\phi_2 \in \mathcal{C}^\infty(M)$, we have: $\esp{\prsc{W}{\phi_1}_{(\D',\mathcal{C}^\infty)}\prsc{W}{\phi_2}_{(\D',\mathcal{C}^\infty)}}= \displaystyle\int_M \phi_1 \phi_2 \rmes{M}$.
\end{itemize}
In particular, we have $\prsc{W}{\phi}_{(\D',\mathcal{C}^\infty)} \sim \mathcal{N}(0,\Norm{\phi}^2_2)$, where $\Norm{\phi}_2$ stands for the $L^2$-norm of $\phi$.
\end{dfn}

Here and in the following, we avoid using the term ``distribution'' when talking about elements of $\D'(M)$ and rather use the term ``generalized function''. This is to avoid any possible confusion with the distribution of a random variable.

\begin{thm}[Central Limit Theorem]
\label{thm central limit theorem}
Let $\X$ be a real projective curve whose real locus~$M$ is non-empty. Let $\E \to \X$ and $\L \to \X$ be real Hermitian line bundles such that $\L$ is positive. Let $(s_d)_{d \geq 1}$ be a sequence of standard Gaussian vectors in $\prod_{d \geq 1}\R \H$. For any $d \geq 1$, let $Z_d$ denote the real zero set of $s_d$ and let $\nu_d$ denote the counting measure of $Z_d$.

Then, the following holds in distribution, in the space $\mathcal{D}'(M)$ of generalized functions on $M$:
\begin{equation*}
\frac{1}{d^\frac{1}{4}\sigma}\left(\nu_d - \frac{d^\frac{1}{2}}{\pi}\rmes{M}\right) \xrightarrow[d \to +\infty]{} W,
\end{equation*}
where $W$ denotes the Standard Gaussian White Noise on $M$ (see~Definition~\ref{def W}).

In particular, for any test-functions $\phi_1,\dots,\phi_k \in \mathcal{C}^\infty(M)$, the random vector:
\begin{equation*}
\frac{1}{d^\frac{1}{4}\sigma}\left(\left(\prsc{\nu_d}{\phi_1} - \frac{d^\frac{1}{2}}{\pi}\int_M \phi_1\rmes{M}\right), \dots, \left(\prsc{\nu_d}{\phi_k} - \frac{d^\frac{1}{2}}{\pi}\int_M \phi_k\rmes{M}\right)\right)
\end{equation*}
converges in distribution to a centered Gaussian vector of variance matrix $\begin{pmatrix}
\displaystyle\int_M \phi_i \phi_j \rmes{M}
\end{pmatrix}_{1 \leq i,j \leq k}$.

Moreover, for all $\phi \in \mathcal{C}^0(M)$,
\begin{equation*}
\frac{1}{d^\frac{1}{4}\sigma}\left(\prsc{\nu_d}{\phi} - \frac{d^\frac{1}{2}}{\pi}\int_M \phi\rmes{M}\right) \xrightarrow[d \to +\infty]{} \mathcal{N}(0,\Norm{\phi}_2^2)
\end{equation*}
in distribution. In particular, $\left(d^\frac{1}{4}\sigma\vol{M}^{\frac{1}{2}}\right)^{-1}\left(\card(Z_d) - \frac{d^\frac{1}{2}}{\pi} \vol{M}\right) \xrightarrow[d \to +\infty]{} \mathcal{N}(0,1)$.
\end{thm}

Before stating our main result, we need to introduce some additional notations.

\begin{dfn}[Moments of the standard Gaussian]
\label{def mu p}
For all $p \in \N$, we denote by $\mu_p$ the $p$-th moment of the standard real Gaussian distribution. Recall that, for all $p \in \N$, we have $\mu_{2p} = \frac{(2p)!}{2^p p!}$ and $\mu_{2p+1}=0$.
\end{dfn}

\begin{dfn}[Partitions]
\label{def partitions}
Let $A$ be a finite set, a \emph{partition} of $A$ is a family $\I$ of non-empty disjoint subsets of $A$ such that $\bigsqcup_{I \in \I} I =A$. We denote by $\pa_A$ (resp.~$\pa_k$) the set of partitions of~$A$ (resp.~$\{1,\dots,k\}$). A \emph{partition into pairs} of $A$ is a partition $\I \in \pa_A$ such that $\card(I)=2$ for all $I \in \I$. We denote by $\pp_A$, (resp.~$\pp_k$) the set of partitions into pairs of $A$ (resp.~$\{1,\dots,k\}$). Note that $\pa_\emptyset = \{ \emptyset \} = \pp_\emptyset$.
\end{dfn}

\begin{thm}[Central moments asymptotics]
\label{thm main}
Let $\X$ be a real projective curve whose real locus~$M$ is non-empty. Let $\E \to \X$ and $\L \to \X$ be real Hermitian line bundles such that $\L$ is positive. For all $d \in \N$, let $s_d \in \R\H$ be a standard Gaussian vector, let $Z_d$ denote the real zero set of $s_d$ and let $\nu_d$ denote its counting measure. For all $p \geq 3$, for all $\phi_1,\dots,\phi_p \in \mathcal{C}^0(M)$, the following holds as $d \to +\infty$:
\begin{align*}
m_p(\nu_d)(\phi_1,\dots,\phi_p) &= \sum_{\I \in \pp_p} \prod_{\{i,j\} \in \I} m_2(\nu_d)(\phi_i,\phi_j) + \left(\prod_{i=1}^p \Norm{\phi_i}_\infty\right) O\!\left(d^{\frac{1}{2}\floor*{\frac{p-1}{2}}}(\ln d)^p\right)\\
&= d^\frac{p}{4}\sigma^p \sum_{\I \in \pp_p} \prod_{\{i,j\} \in \I} \left(\int_M \phi_i\phi_j \rmes{M}\right) + o(d^\frac{p}{4}),
\end{align*}
where $\floor*{\cdot}$ denotes the integer part, $\Norm{\cdot}_\infty$ denotes the sup-norm, $\sigma$ is the same positive constant as in Theorem~\ref{thm variance}, the set $\pp_p$ is defined by Definition~\ref{def partitions}, and the error term $O\!\left(d^{\frac{1}{2}\floor*{\frac{p-1}{2}}}(\ln d)^p\right)$ does not depend on $(\phi_1,\dots,\phi_p)$.

In particular, for all $\phi \in \mathcal{C}^0(M)$, we have:
\begin{align*}
m_p\left(\prsc{\nu_d}{\phi}\right) &= \mu_p \var{\prsc{\nu_d}{\phi}}^\frac{p}{2} + \Norm{\phi}_\infty^p O\!\left(d^{\frac{1}{2}\floor*{\frac{p-1}{2}}}(\ln d)^p\right)\\
&= \mu_p d^\frac{p}{4}\sigma^p \left(\int_M \phi^2 \rmes{M}\right)^\frac{p}{2} + o(d^\frac{p}{4}).
\end{align*}
\end{thm}

\begin{rem}
\label{rem main thm}
If $p$ is odd, then the first term vanishes in the asymptotic expansions of Theorem~\ref{thm main}. Indeed, in this case $\pp_p = \emptyset$ and $\mu_p = 0$. Hence, if $p$ is odd, for all $\phi_1,\dots,\phi_p \in \mathcal{C}^0(M)$, we have $m_p(\nu_d)(\phi_1,\dots,\phi_p) = O(d^{\frac{p-1}{4}}(\ln d)^p)$. If $p$ is even, we have $m_p(\nu_d)(\phi_1,\dots,\phi_p) = O(d^{\frac{p}{4}})$ for all $\phi_1,\dots,\phi_p \in \mathcal{C}^0(M)$.
\end{rem}

Other interesting corollaries of Theorem~\ref{thm main} include the following.

\begin{cor}[Concentration in probability]
\label{cor concentration}
In the setting of Theorem~\ref{thm main}, let $(\epsilon_d)_{d \geq 1}$ denote a sequence of positive numbers and let $\phi \in \mathcal{C}^0(M)$. Then, for all $p \in \N^*$, as $d \to +\infty$, we have:
\begin{equation*}
\P\left(d^{-\frac{1}{2}}\norm{\prsc{\nu_d}{\phi} - \rule{0pt}{10pt}\esp{\prsc{\nu_d}{\phi}}} > \epsilon_d \right) = O\left((d^\frac{1}{4}\epsilon_d)^{-2p}\right).
\end{equation*}
In particular, for all $p \in \N^*$, as $d \to +\infty$, we have:
\begin{equation*}
\P\left(d^{-\frac{1}{2}}\norm{\card(Z_d) - \rule{0pt}{10pt}\esp{\card(Z_d)}} > \epsilon_d \right) = O\left((d^\frac{1}{4}\epsilon_d)^{-2p}\right).
\end{equation*}
\end{cor}

Corollary~\ref{cor concentration} and Theorem~\ref{thm expectation} imply that $\P(\card(Z_d) > \sqrt{d}C)=O(d^{-\frac{p}{2}})$ for any $p \in \N^*$ and $C > \frac{1}{\pi}\vol{M}$. In the same spirit, Gayet and Welschinger proved in~\cite[Theorem~2]{GW2011} that there exists $D>0$ such that $\P( \card(Z_d) > \sqrt{d}\epsilon_d)=O\!\left(\frac{\sqrt{d}}{\epsilon_d}e^{-D\epsilon_d^2}\right)$ for any positive sequence $(\epsilon_d)_{d \geq 1}$ such that $\epsilon_d \xrightarrow[d \to +\infty]{} +\infty$. In the other direction, the following corollary bounds the probability that $Z_d$ is empty.

\begin{cor}[Hole probability]
\label{cor hole probability}
In the setting of Theorem~\ref{thm main}, let $U$ be a non-empty open subset of $M$. Then, for all $p \in \N^*$, as $d \to +\infty$, we have $\P\left(Z_d \cap U = \emptyset\right) = O(d^{-\frac{p}{2}})$.
\end{cor}

\paragraph{About the proofs.}

The proof of Theorem~\ref{thm main} relies on several ingredients. Some of them are classical, such as Kac--Rice formulas and estimates for the Bergman kernel of $\E \otimes \L^d$, other are new, such as the key combinatorial argument that we develop in Section~\ref{sec asymptotics of the central moments}.

Kac--Rice formulas are a classical tool in the study of the number of real roots of random polynomials (see~\cite{AT2007,AW2009} for example). More generally, they allow to express the moments of local quantities associated with the level sets of a Gaussian process, such as their volume or their Euler characteristic, only in terms of the correlation function of the process. Even if these formulas are well-known, it is the first time, to the best of our knowledge, that they are used to compute the exact asymptotics of central moments of order greater than $3$. The Kac--Rice formulas we use in this paper were proved in~\cite{Anc2021}. We recall them in Proposition~\ref{prop Kac Rice}. They allow us to write the $p$-th central moment $m_p(\nu_d)(\phi_1,\dots,\phi_p)$ as the integral over $M^p$ of $\phi:(x_1,\dots,x_p) \mapsto \prod_{i=1}^p \phi_i(x_i)$ times some density function $\D_d^p$. Here we are cheating a bit: the random set $Z_d$ being almost surely discrete, the Kac--Rice formulas yield the so-called factorial moments of $\nu_d$ instead of the usual moments. This issue is usual (compare~\cite{AT2007,AW2009}), and it will not trouble us much since the central moments can be written as linear combinations of the factorial ones. For the purpose of this sketch of proof, let us pretend that we have indeed:
\begin{equation*}
m_p(\nu_d)(\phi_1,\dots,\phi_p) = \int_{M^p} \phi \D_d^p \rmes{M}^p.
\end{equation*}
This simplified situation is enough to understand the main ideas of the proof. The correct statement is given in Lemma~\ref{lem integral expression} below.

The density $\D_d^p$ is a polynomial in the Kac--Rice densities $(\mathcal{R}^k	_d)_{1 \leq k \leq p}$ appearing in Definition~\ref{def Rkd}. As such, it only depends on the correlation function of the Gaussian process $(s_d(x))_{x \in M}$, which is the Bergman kernel of~$\E \otimes \L^d$. This kernel admits a universal local scaling limit at scale~$d^{-\frac{1}{2}}$, which is exponentially decreasing (cf.~\cite{MM2007,MM2015}). In~\cite{Anc2021}, the author used these Bergman kernel asymptotics and Olver multispaces (see~\cite{Olv2001}) to prove estimates for the $(\mathcal{R}^k	_d)_{1 \leq k \leq p}$ in the large~$d$ limit. These key estimates are recalled in Propositions~\ref{prop bounded density} and~\ref{prop multiplicativity Rk} below. They allow us to show that $\D_d^p(x) = O(d^\frac{p}{2})$, uniformly in $x \in M^p$. Moreover, we show that $\D_d^p(x) = O(d^{\frac{p}{4}-1})$, uniformly in $x \in M^p$ such that one of the components of $x$ is far from the others. By this we mean that $x =(x_1,\dots,x_p)$ and there exists $i \in \{1,\dots,p\}$ such that, for all $j \neq i$, the distance from $x_i$ to $x_j$ is larger than $b_p\frac{\ln d}{\sqrt{d}}$, where $b_p >0$ is some well-chosen constant.

In order to understand the integral of $\phi\D^p_d$, we split $M^p$ as follows. For $x=(x_1,\dots,x_p) \in M^p$, we define a graph (see Definition~\ref{def Gdp}) whose vertices are the integers $\{1,\dots,p\}$, with an edge between $i$ and $j$ if and only if $i \neq j$ and the distance from $x_i$ to $x_j$ is less than $b_p \frac{\ln d}{\sqrt{d}}$. The connected components of this graph yield a partition $\I(x) \in \pa_p$ (see Definition~\ref{def Idp}) encoding how the $(x_i)_{1\leq i \leq p}$ are clustered in~$M$, at scale $d^{-\frac{1}{2}}$. An example of this construction is represented on Figure~\ref{fig clusters} in Section~\ref{subsec cutting MA into pieces}. Denoting by $M^p_\I = \{ x \in M^p \mid \I(x)=\I\}$, we have:
\begin{equation*}
m_p(\nu_d)(\phi_1,\dots,\phi_p) = \sum_{\I \in \pa_p} \int_{M^p_\I} \phi \D_d^p \rmes{M}^p.
\end{equation*}
Thanks to our estimates on $\D^p_d$, we show that if $\I$ contains a singleton then the integral over $M^p_\I$ is $O(d^{\frac{p}{4}-1})$. Hence it contributes only an error term in Theorem~\ref{thm main}. Moreover, denoting by $\norm{\I}$ the cardinality of~$\I$ (i.e. the number of clusters), the volume of $M^p_\I$ is $O(d^\frac{\norm{\I}-p}{2}(\ln d)^{p})$. Hence, the integral over~$M^p_\I$ is $O(d^\frac{\norm{\I}}{2}(\ln d)^p)$. If $\norm{\I} < \frac{p}{2}$, this is also an error term in Theorem~\ref{thm main}. Thus, the main contribution in $m_p(\nu_d)(\phi_1,\dots,\phi_p)$ comes from the integral of $\phi \D^p_d$ over the pieces $M^p_\I$ indexed by partitions $\I \in \pa_p$ without singletons and such that $\norm{\I} \geq \frac{p}{2}$. These are exactly the partitions into pairs of $\{1,\dots,p\}$. Finally, if $\I\in \pp_p$, we prove the contribution of the integral over $M^p_\I$ to be equal to the product of covariances $\prod_{\{i,j\} \in \I} m_2(\nu_d)(\phi_i,\phi_j)$, up to an error term. When all the test-functions $(\phi_i)_{1 \leq i \leq p}$ are equal, the $p$-th moment $\mu_p$ of the standard Gaussian distribution appears as the cardinality of the set of partitions of $\{1,\dots,p\}$ into pairs.

Concerning the corollaries, Corollaries~\ref{cor concentration} and~\ref{cor hole probability} follow from Theorem~\ref{thm main} and Markov's Inequality for the $2p$-th moment. The strong Law of Large Numbers (Theorem~\ref{thm law of large numbers}) is deduced from Theorem~\ref{thm main} for $p=6$ by a Borel--Cantelli type argument. The Central Limit Theorem (Theorem~\ref{thm central limit theorem}) for the linear statistics is obtained by the method of moments. The functional version of this Central Limit Theorem is then obtained by the Lévy--Fernique Theorem (cf.~\cite{Fer1967}), which is an extension of Lévy's Continuity Theorem adapted to generalized random processes.

\paragraph{Higher dimension.}

In this paper, we are concerned with the real roots of a random polynomial (or a random section) in an ambient space of dimension $1$. There is a natural higher dimensional analogue of this problem. Namely, one can consider the common zero set $Z_d \subset \R P^n$ of $r$ independent real Kostlan polynomials in $n+1$ variables, where $r \in \{1,\dots,n\}$. More generally, we consider the real zero set $Z_d$ of a random real section of $\E \otimes \L^d \to \X$ in the complex Fubini--Study model, where $\X$ is a real projective manifold of complex dimension~$n$ whose real locus $M$ is non-empty, $\L$ is a positive line bundle as above, and $\E$ is a rank~$r$ real Hermitian bundle with $1 \leq r \leq n$. Then, for $d$ large enough, $Z_d$ is almost surely a smooth closed submanifold of codimension $r$ in the smooth closed $n$-dimensional manifold $M$. In this setting, $M$ is equipped with a natural Riemannian metric that induces a volume measure $\rmes{M}$ on $M$ and a volume measure~$\nu_d$ on $Z_d$. As in the $1$-dimensional case, $\nu_d$ is an almost surely well-defined random Radon measure on $M$. In this higher dimensional setting, we have the following analogues of Theorem~\ref{thm expectation} and~\ref{thm variance} (see~\cite{Let2016,Let2019,LP2019}):
\begin{align*}
&\forall \phi \in \mathcal{C}^0(M), & & \esp{\prsc{\nu_d}{\phi}} = d^\frac{r}{2} \frac{\vol{\S^{n-r}}}{\vol{\S^n}} \int_{M}\phi \rmes{M} + \Norm{\phi}_\infty O(d^{\frac{r}{2}-1}),\\
&\forall \phi_1,\phi_2 \in \mathcal{C}^0(M), & & m_2(\nu_d)(\phi_1,\phi_2) = d^{r -\frac{n}{2}} \sigma_{n,r}^2 \int_M \phi_1\phi_2 \rmes{M} + o(d^{r-\frac{n}{2}}),
\end{align*}
where $\sigma_{n,r}>0$ is a universal constant depending only on $n$ and $r$. In~\cite{LP2019}, Letendre and Puchol proved some analogues of Corollaries~\ref{cor concentration} and~\ref{cor hole probability} for any $n$ and $r$. They also showed that the strong Law of Large Numbers (Theorem~\ref{thm law of large numbers}) holds if $n \geq 3$.

Most of the proof of Theorem~\ref{thm main} is valid in any dimension and codimension. In fact, the combinatorics are simpler when $r<n$. The only things we are missing, in order to prove the analogue of Theorem~\ref{thm main} for any $n$ and $r$, are higher dimensional versions of Propositions~\ref{prop bounded density} and~\ref{prop multiplicativity Rk}. The proofs of these propositions (see~\cite{Anc2021}) rely on the compactness of Olver multispaces, which holds when $n=1$ but fails for $n>1$. This seems to be only a technical obstacle and the authors are currently working toward the following.

\begin{conj}
Let $s_d$ be a random section in the complex Fubini--Study model in dimension $n$ and codimension $r \in \{1,\dots,n\}$. Let $\nu_d$ denote the volume measure of the real zero set of $s_d$. For all $p \geq 3$, for all $\phi_1,\dots,\phi_p \in \mathcal{C}^0(M)$, the following holds as $d \to +\infty$:
\begin{equation*}
m_p(\nu_d)(\phi_1,\dots,\phi_p) = d^{\frac{p}{2}(r-\frac{n}{2})}\sigma_{n,r}^p \sum_{\I \in \pp_p} \prod_{\{i,j\} \in \I} \left(\int_M \phi_i\phi_j \rmes{M}\right) + o(d^{\frac{p}{2}(r-\frac{n}{2})}).
\end{equation*}
In particular, for all $\phi \in \mathcal{C}^0(M)$, $\displaystyle m_p(\prsc{\nu_d}{\phi})= \mu_p d^{\frac{p}{2}(r-\frac{n}{2})}\sigma_{n,r}^p \left(\int_M \phi^2 \rmes{M}\right)^\frac{p}{2} + o(d^{\frac{p}{2}(r-\frac{n}{2})})$.
\end{conj}

Proving this conjecture for $n=2$ and $p=4$ is enough to prove that the strong Law of Large Numbers (Theorem~\ref{thm law of large numbers}) holds for $n=2$, which is the only missing case. This conjecture also implies the Central Limit Theorem (Theorem~\ref{thm central limit theorem}) in dimension $n$ and codimension $r$, with the same proof as the one given in Section~\ref{subsec proof of the CLT}. Note that a Central Limit Theorem for the volume of the common zero set of $r$ Kostlan polynomials in $\R P^n$ was proved in~\cite{AADL2018,AADL2021}.

\paragraph{Other related works.}

The complex roots of complex Kostlan polynomials have been studied in relation with Physics in~\cite{BBL1992}. More generally, complex zeros of random holomorphic sections of positive line bundles over projective manifolds were studied in~\cite{SZ1999} and some subsequent papers by the same authors. In~\cite{SZ1999}, they computed the asymptotics of the expected current of integration over the complex zeros of such a random section, and proved a Law of Large Numbers similar to Theorem~\ref{thm law of large numbers}. In~\cite{SZ2010}, they obtained a variance estimate for this random current, and proved that it satisfies a Central Limit Theorem. This last paper extends the results of~\cite{ST2004} for the complex roots of a family of random polynomials, including elliptic ones. In~\cite{BSZ2000}, Bleher, Shiffman and Zelditch studied the $p$-points zero correlation function associated with random holomorphic sections. These functions are the Kac--Rice densities for the non-central $p$-th moment of the linear statistics in the complex case. The results of~\cite{Anc2021} are also valid in the $1$-dimensional complex case, see~\cite[Section~5]{Anc2021}. Thus, Theorem~\ref{thm main} can be extended to the complex case, with the same proof.

In Corollaries~\ref{cor concentration} and~\ref{cor hole probability}, we deduce from Theorem~\ref{thm main} some concentration in probability, faster than any negative power of $d$. However, our results are not precise enough to prove that this concentration is exponentially fast in $d$. In order to obtain such a large deviation estimate, one would need to investigate how the constants involved in the error terms in Theorem~\ref{thm main} grow with~$p$. Some large deviations estimates are known for the complex roots of random polynomials. As far as real roots are concerned, the only result of this kind we are aware of is~\cite{BDFZ2018}.

The Central Limit Theorem (Theorem~\ref{thm central limit theorem}) was already known for the roots of Kostlan polynomials, see~\cite{Dal2015}. In the wider context of random real geometry, Central Limit Theorems are known in several settings, see~\cite{AADL2018,AADL2021,Ros2019} and the references therein. The proofs of all these results rely on Wiener chaos techniques developed by Kratz--Le\`on~\cite{KL2001}. Our proof of Theorem~\ref{thm central limit theorem} follows a different path, relying on the method of moments, for two reasons. First, to the best of our knowledge, Wiener chaos techniques are not available for random sections of line bundles. Second, these techniques are particularly convenient when dealing with random models with slowly decaying correlations, for example Random Waves models. In these cases, one of the first chaotic components, usually the second or the fourth, is asymptotically dominating. Hence, one can reduce the study of, say, the number of zeros to that of its dominating chaotic component, which is easier to handle. In the complex Fubini--Study setting we are considering in this paper, the correlations are exponentially decreasing, so that all chaotic components of the number of zeros have the same order of magnitude as the degree goes to infinity. In order to use Wiener chaoses to prove a Central Limit Theorem in our setting, one would have to study the joint asymptotic behavior of all the chaotic components, which seems more difficult than our method.

For real zeros in ambient dimension $n=1$, Nazarov and Sodin~\cite{NS2016} proved a strong Law of Large Numbers, as $R \to +\infty$, for the number of zeros of a Gaussian process lying in the interval $[-R,R]$. Finally, to the best of our knowledge, Theorem~\ref{thm main} gives the first precise estimate for the central moments of the number of real zeros of a family of random processes.

\paragraph{Organization of the paper.}

This paper is organized as follows. In Section~\ref{sec framework and background} we introduce the object of our study and recall some useful previous results. More precisely, we introduce our geometric framework in Section~\ref{subsec geometric setting}. In Section~\ref{subsec partitions products and diagonal inclusions} we introduce various notations that will allow us to make sense of the combinatorics involved in our problem. The random measures we study are defined in Section~\ref{subsec zeros of random real sections}. Finally, we state the Kac--Rice formulas for higher moments in Section~\ref{subsec Kac-Rice formulas and density functions} and we recall several results from~\cite{Anc2021} concerning the density functions appearing in these formulas. In Section~\ref{sec asymptotics of the central moments}, we prove our main result, that is the moments estimates of Theorem~\ref{thm main}. Section~\ref{sec proof of the corollaries} is concerned with the proofs of the corollaries of Theorem~\ref{thm main}. We prove the Law of Large Numbers (Theorem~\ref{thm law of large numbers}) in Section~\ref{subsec proof of the LLN}, the Central Limit Theorem (Theorem~\ref{thm central limit theorem}) in Section~\ref{subsec proof of the CLT} and the remaining corollaries (Corollaries~\ref{cor concentration} and~\ref{cor hole probability}) in Section~\ref{subsec proof of the corollaries}.

\paragraph{Acknowledgments.}

Thomas Letendre thanks Hugo Vanneuville for a useful discussion about the method of moments. He is also grateful to Julien Fageot for inspiring conversations about generalized random processes. The authors thank Damien Gayet and Jean-Yves Welschinger for their unwavering support. 


\section{Framework and background}
\label{sec framework and background}

We start this section by defining precisely the geometric setting in which we work. This is the purpose of Section~\ref{subsec geometric setting}. In Section~\ref{subsec partitions products and diagonal inclusions}, we introduce the counting measures that we study and explain how they can be splitted in terms that are easier to handle. Section~\ref{subsec zeros of random real sections} is dedicated to the definition of the model of random sections we consider. Finally, we recall the Kac--Rice formulas we need in Section~\ref{subsec Kac-Rice formulas and density functions}, as well as several useful estimates for the density functions appearing in these formulas.


\subsection{Geometric setting}
\label{subsec geometric setting}

In this section, we introduce our geometric framework, which is the same as that of~\cite{Anc2021,GW2016,Let2016,Let2019,LP2019}. See also~\cite{BSZ2000,SZ1999,SZ2010}, where the authors work in a related complex setting. A classical reference for some of the material of this section is~\cite{GH1994}.

\begin{itemize}
\item Let $(\X,c_{\X})$ be a smooth real projective curve, that is a smooth complex projective manifold $\X$ of complex dimension $1$, equipped with an anti-holomorphic involution $c_{\X}$. We denote by
$M$ the real locus of $\X$, that is the set of fixed points of $c_\X$. Throughout the paper, we assume that $M$ is not empty. In this case, $M$ is a smooth compact submanifold of $\X$ of real dimension $1$ without boundary, that is the disjoint union of a finite number of circles.

\item Let $(\E,c_\E)$ and $(\L,c_\L)$ be two real holomorphic line bundles over $(\X,c_\X)$. Denoting by $\pi_\E$ (resp.~$\pi_\L$) the bundle projection, this means that $\E \to \X$ (resp.~$\L \to \X$) is an holomorphic line bundle such that $\pi_\E \circ c_\E = c_\X \circ \pi_\E$ (resp.~$\pi_\L \circ c_\L = c_\X \circ \pi_\L$) and that $c_\E$ (resp.~$c_\L$) is anti-holomorphic and fiberwise $\C$-anti-linear. For any $d \in \N$, we denote by $c_d = c_\E \otimes c_\L^d$. Then, $(\E \otimes \L^d,c_d)$ is a real holomorphic line bundle over $(\X,c_\X)$.

\item We equip $\E$ with a real Hermitian metric $h_\E$. That is $(\E,h_\E)$ is an Hermitian line bundle, and moreover $c_\E^*h_\E = \bar{h_\E}$. Similarly, let $h_\L$ denote a real Hermitian metric on $\L$. For all $d \in \N$, we denote by $h_d = h_\E \otimes h_\L^d$, which defines a real Hermitian metric on $\E \otimes \L^d \to \X$. We assume that $(\L,h_\L)$ is positive, in the sense that its curvature form $\omega$ is a Kähler form. Recall that $\omega$ is locally defined as $\frac{1}{2i}\partial\debar \ln(h_\L(\zeta,\zeta))$, where $\zeta$ is any local holomorphic frame of $\L$. The Kähler structure defines a Riemannian metric $g = \omega(\cdot,i\cdot)$ on $\X$, hence on $M$. The Riemannian volume form on $\X$ associated with $g$ is simply $\omega$. We denote by $\rmes{M}$ the arc-length measure on $M$ associated with~$g$. For all $k \in \N^*$, we denote by $\rmes{M}^k$ the product measure on $M^k$.

\item For any $d \in \N$, we denote by $\H$ the space of global holomorphic sections of $\E \otimes \L^d$. This is a complex vector space of complex dimension $N_d$. By the Riemann--Roch Theorem, $N_d$ is finite and diverges to infinity as $d \to +\infty$. We denote by:
\begin{equation*}
\R\H = \left\{s \in \H \mvert s \circ c_\X = c_d \circ s \right\}
\end{equation*}
the space of global real holomorphic sections of $\E \to \L^d$, which is a real vector space of real dimension $N_d$. Let $x \in M$, the fiber $(\E \otimes \L^d)_x$ is a complex line equipped with a $\C$-anti-linear involution $c_d(x)$. We denote by $\R(\E \otimes \L^d)_x$ the set of fixed points of $c_d(x)$, which is a real line. Then, $\R(\E \otimes \L^d) \to M$ is a real line bundle and, for any $s \in \R\H$, the restriction of $s$ to $M$ is a smooth section of $\R(\E \otimes \L^d) \to M$.

\item The volume form $\omega$ and the Hermitian metric $h_d$ induce an Hermitian $L^2$-inner product $\prsc{\cdot}{\cdot}$ on $\H$. It is defined by:
\begin{equation}
\label{eq def L2 inner product}
\forall s_1,s_2 \in \H, \qquad \prsc{s_1}{s_2}=\int_\X h_d(s_1,s_2) \omega.
\end{equation}
The restriction of this inner product to $\R\H$ is a Euclidean inner product.

\item For any section $s \in \R\H$, we denote by $Z_s = s^{-1}(0) \cap M$ its real zero set. Since~$s$ is holomorphic, if $s \neq 0$ its zeros are isolated. In this case, $Z_s$ is finite by compactness of $M$, and we denote by $\nu_s = \sum_{x \in Z_s} \delta_x$, where $\delta_x$ stands for the unit Dirac mass at $x \in M$. The measure~$\nu_s$ is called the \emph{counting measure} of $Z_s$. It is a Radon measure, that is a continuous linear form on the space $(\mathcal{C}^0(M), \Norm{\cdot}_\infty)$ of continuous functions equipped with the sup-norm. It acts on continuous functions by: for all $\phi \in \mathcal{C}^0(M)$, $\prsc{\nu_s}{\phi} = \sum_{x \in Z_s} \phi(x)$.
\end{itemize}

\begin{exa}[Kostlan scalar product]
\label{ex Kostlan inner product}
We conclude this section by giving an example of our geometric setting. We consider $\X =\C P^1$, equipped with the conjugation induced by the one in~$\C^2$. Its real locus is $M=\R P^1$. We take $\E$ to be trivial and $\L$ to be the dual of the tautological line bundle $\left\{(v,x) \in \C^2 \times \C P^1 \mvert v \in x \right\}$, that is $\L = \mathcal{O}(1)$. Both $\E$ and $\L$ are canonically real Hermitian line bundle over $\C P^1$, and the curvature of $\L$ is the Fubini--Study form, normalized so that $\vol{\C P^1}=\pi$. The corresponding Riemannian metric on~$\R P^1$ is the quotient of the metric on the Euclidean unit circle, so that the length of $\R P^1$ is~$\pi$.

In this setting, $\H$ (resp.~$\R\H$) is the space of homogeneous polynomials of degree $d$ in two variables with complex (resp.~real) coefficients. If $s \in \R\H$ is such a polynomial, then $Z_s$ is the set of its roots in $\R P^1$. Finally, up to a factor $(d+1)\pi$ which is irrelevant to us, the inner product of Equation~\eqref{eq def L2 inner product} is defined by:
\begin{equation*}
\prsc{P}{Q} = \frac{1}{\pi^2d!} \int_{\C^2} P(z) \bar{Q(z)} e^{-\Norm{z}^2} \dx z,
\end{equation*}
for any homogeneous polynomials $P$ and $Q$ of degree $d$ in two variables. In particular, the family $\left\{\sqrt{\binom{d}{k}}X^kY^{d-k} \mvert 0\leq k\leq d\right\}$ is an orthonormal basis of $\R\H$ for this inner product.
\end{exa}


\subsection{Partitions, products and diagonal inclusions}
\label{subsec partitions products and diagonal inclusions}

In this section, we introduce some notations that will be useful throughout the paper, in particular to sort out the combinatorics involved in the proof of Theorem~\ref{thm main} (see Section~\ref{sec asymptotics of the central moments}). In all this section, $M$ denotes a smooth manifold.

\begin{ntns}
\label{ntn product indexed by A}
Let $A$ be a finite set.
\begin{itemize}
\item We denote by $\card(A)$ or by $\norm{A}$ the cardinality of $A$.
\item We denote by $M^A$ the Cartesian product of $\norm{A}$ copies of $M$, indexed by the elements of $A$.
\item A generic element of $M^A$ is denoted by $\underline{x}_A =(x_a)_{a \in A}$. If $B \subset A$ we denote by $\underline{x}_B =(x_a)_{a \in B}$.
\item Let $(\phi_a)_{a \in A}$ be continuous functions on $M$, we denote by $\phi_A = \boxtimes_{a\in A}\phi_a$ the function on $M^A$ defined by: $\phi_A(\underline{x}_A)=\prod_{a\in A}\phi_a(x_a)$, for all $\underline{x}_A = (x_a)_{a \in A} \in M^A$.
\end{itemize}
If $A$ is clear from the context or of the form $\{1,\dots,k\}$ with $k \in \N^*$, we use the simpler notations $x$ for $\underline{x}_A$ and $\phi$ for $\phi_A$.
\end{ntns}

Recall the we defined the set $\pa_A$ (resp.~$\pa_k$) of partitions of a finite set $A$ (resp. of $\{1,\dots,k\}$) in the introduction, see Definition~\ref{def partitions}.

\begin{dfn}[Diagonals]
\label{def diagonals}
Let $A$ be a finite set, we denote by $\Delta_A$ the \emph{large diagonal} of $M^A$, that is:
\begin{equation*}
\Delta_A = \left\{(x_a)_{a \in A}\in M^A \mvert \rule{0pt}{10pt} \exists a,b \in A \text{ such that } a \neq b \text{ and } x_a = x_b \right\}.
\end{equation*}
Moreover, for all $\I \in \pa_A$, we denote by
\begin{equation*}
\Delta_{A,\I} = \left\{(x_a)_{a \in A}\in M^A \mvert \rule{0pt}{10pt} \forall a,b \in A, \left(x_a=x_b \iff \exists I \in \I \text{ such that } a \in I \text{ and } b \in I \right)\right\}.
\end{equation*}
If $A = \{1,\dots,k\}$, we use the simpler notations $\Delta_k$ for $\Delta_A$, and $\Delta_{k,\I}$ for $\Delta_{A,\I}$.
\end{dfn}

\begin{dfn}[Diagonal inclusions]
\label{def diagonal inclusions}
Let $A$ be a finite set and let $\I \in \pa_A$. The \emph{diagonal inclusion} $\iota_\I:M^\I\rightarrow M^A$ is the function defined, for all $\underline{y}_\I= (y_I)_{I \in \I} \in M^\I$, by $\iota_\I(\underline{y}_\I)= (x_a)_{a\in A}$, where for all $I \in \I$, for all $a \in I$, we set $x_a = y_I$.
\end{dfn}

\begin{rem}
\label{rem diagonal inclusions}
With these definitions, we have $M^A = \bigsqcup_{\I \in \pa_A} \Delta_{A,\I}$ and $\Delta_A = \bigsqcup_{\I \in \pa_A \setminus \{\I_0(A)\}} \Delta_{A,\I}$, where we denoted by $\I_0(A)= \left\{ \{a\} \mvert a \in A\right\}$. Moreover, $\iota_\I$ is a smooth diffeomorphism from $M^\I \setminus \Delta_\I$ onto $\Delta_{A,\I} \subset M^A$. Note that $\Delta_{A,\I_0(A)}$ is the configuration space $M^A \setminus \Delta_A$ of $\norm{A}$ distinct points in~$M$. In the following, we avoid using the notation $\Delta_{A,\I_0(A)}$ and use $M^A \setminus \Delta_A$ instead.
\end{rem}

Let us now go back to the setting of Section~\ref{subsec geometric setting}, in which $M$ is the real locus of the projective manifold $\X$. Let $d \in \N$ and let $s \in \R\H \setminus \{0\}$. In Section~\ref{subsec geometric setting}, we defined the counting measure $\nu_s$ of the real zero set $Z_s$ of $s$. More generally, for any finite set $A$, we can define the counting measure of $Z_s^A = \left\{(x_a)_{a \in A} \in M^A \mvert \forall a \in A, x_a \in Z_s \right\}$ and that of $Z_s^A \setminus \Delta_A$. The latter is especially interesting for us, since this is the one that appears in the Kac--Rice formulas, see Proposition~\ref{prop Kac Rice} below.

\begin{dfn}[Counting measures]
\label{def nu A}
Let $d \in \N$ and let $A$ be a finite set. For any non-zero $s \in \R\H$, we denote~by:
\begin{align*}
\nu^A_s &= \sum_{x \in Z_s^A} \delta_x & &\text{and} & \tilde{\nu}^A_s &= \sum_{x \in Z_s^A \setminus \Delta_A} \delta_x,
\end{align*}
where $\delta_x$ is the unit Dirac mass at $x =(x_a)_{a \in A} \in M^A$ and $\Delta_A$ is defined by Definition~\ref{def diagonals}. These measures are the counting measures of $Z_s^A$ and $Z_s^A \setminus \Delta_A$ respectively. Both $\nu^A_s$ and $\tilde{\nu}^A_s$ are Radon measure on $M^A$. They act on $\mathcal{C}^0(M^A)$ as follows: for any $\phi \in \mathcal{C}^0(M^A)$,
\begin{align*}
\prsc{\nu^A_s}{\phi} &= \sum_{x \in Z_s^A} \phi(x) & &\text{and} & \prsc{\tilde{\nu}^A_s}{\phi} &= \sum_{x \in Z_s^A \setminus \Delta_A} \phi(x).
\end{align*}
As usual, if $A = \{1,\dots,k\}$, we denote $\nu^k_s$ for $\nu^A_s$ and $\tilde{\nu}^k_s$ for $\tilde{\nu}^A_s$.
\end{dfn}

We have seen in Remark~\ref{rem diagonal inclusions} that $M^A$ splits as the disjoint union of the diagonals $\Delta_{A,\I}$, with $\I \in \pa_A$. Taking the intersection with $Z_s^A$ yields a splitting of this set. Using the diagonal inclusions of Definition~\ref{def diagonal inclusions}, this can be expressed in terms of counting measures as follows.

\begin{lem}
\label{lem decomposition nu ks}
Let $d \in \N$ and let $A$ be a finite set. For any $s \in \R\H \setminus \{0\}$, we have:
\begin{equation*}
\nu^A_s = \sum_{\I \in \pa_A} (\iota_\I)_*\left(\tilde{\nu}^\I_s\right).
\end{equation*}
\end{lem}

\begin{proof}
Recall that $M^A = \bigsqcup_{\I \in \pa_A} \Delta_{A,\I}$. Then, we have:
\begin{equation*}
\nu^A_s = \sum_{x \in Z_s^A} \delta_x = \sum_{\I \in \pa_A} \left(\sum_{x \in Z_s^A \cap \Delta_{A,\I}} \delta_x\right).
\end{equation*}
Let $\I \in \pa_A$, recall that $\iota_\I$ defines a smooth diffeomorphism from $M^\I \setminus \Delta_\I$ onto $\Delta_{A,\I}$. Moreover, $\iota_\I(Z^\I_s \setminus \Delta_\I) = Z^A_s \cap \Delta_{A,\I}$ (see Definition~\ref{def diagonals} and~\ref{def diagonal inclusions}). Hence,
\begin{equation*}
\sum_{x \in Z_s^A \cap \Delta_{A,\I}} \delta_x = \sum_{y \in Z^\I_s \setminus \Delta_\I} \delta_{\iota_\I(y)} = \sum_{y \in Z^\I_s \setminus \Delta_\I} (\iota_\I)_*\delta_y = (\iota_\I)_*\tilde{\nu}^\I_s.\qedhere
\end{equation*}
\end{proof}


\subsection{Zeros of random real sections}
\label{subsec zeros of random real sections}

Let us now introduce the main object of our study: the random Radon measure $\nu_d$ encoding the real zeros of a random real section $s_d \in \R\H$. The model of random real sections we study is often referred to as the complex Fubini--Study model. It was introduced in this generality by Gayet and Welschinger in~\cite{GW2011}. This model is the real counterpart of the model of random holomorphic sections studied by Shiffman and Zelditch in~\cite{SZ1999} and subsequent articles.

\begin{dfn}[Gaussian vectors]
\label{def Gaussian distribution}
Let $(V,\prsc{\cdot}{\cdot})$ be a Euclidean space of dimension $N$, and let $\Lambda$ denote a positive self-adjoint operator on $V$. Recall that a random vector $X$ in $V$ is said to be a \emph{centered Gaussian} with \emph{variance operator} $\Lambda$ if its distribution admits the following density with respect to the normalized Lebesgue measure:
\begin{equation*}
v \mapsto \frac{1}{(2\pi)^\frac{N}{2} \det(\Lambda)^\frac{1}{2}} \exp\left(-\frac{1}{2}\prsc{v}{\Lambda^{-1}v}\right).
\end{equation*}
We denote by $X \sim \mathcal{N}(0,\Lambda)$ the fact that $X$ follows this distribution. If $X \sim \mathcal{N}(0,\Id)$, where $\Id$ is the identity of $V$, we say that $X$ is a \emph{standard Gaussian vector} in $V$.
\end{dfn}

\begin{rem}
Recall that if $(e_1,\dots,e_N)$ is an orthonormal basis of $V$, a random vector $X \in V$ is a standard Gaussian vector if and only if $X = \sum_{i=1}^N a_i e_i$ where the $(a_i)_{1 \leq i \leq N}$ are independent identically distributed $\mathcal{N}(0,1)$ real random variables.
\end{rem}

In the setting of Section~\ref{subsec geometric setting}, for any $d \in \N$, the space $\R\H$ is endowed with the Euclidean inner product $\prsc{\cdot}{\cdot}$ defined by Equation~\eqref{eq def L2 inner product}. We denote by $s_d$ a standard Gaussian vector in $\R\H$. Almost surely $s_d \neq 0$, hence its real zero set and the associated counting measure are well-defined. For simplicity, we denote by $Z_d = Z_{s_d}$ and by $\nu_d = \nu_{s_d}$. Similarly, for any finite set $A$, we denote by $\nu_d^A = \nu^A_{s_d}$ and by $\tilde{\nu}^A_d = \tilde{\nu}^A_{s_d}$ (see Definition~\ref{def nu A}).

\begin{exa}[Kostlan polynomials]
\label{ex Kostlan polynomials}
In the setting of Example~\ref{ex Kostlan inner product}, we have $M=\R P^1$, and $\R\H$ is the set of real homogeneous polynomials of degree $d$ in two variables endowed with the Kostlan inner product. In this case, the standard Gaussian vector $s_d \in \R\H$ is the homogeneous Kostlan polynomial of degree~$d$ defined in Section~\ref{sec introduction}.

Note that the rotations of the circle $M=\R P^1$ are induced by orthogonal transformations of~$\R^2$. The group $O_2(\R)$ acts on the space of homogeneous polynomials in two variables of degree $d$ by composition on the right. Example~\ref{ex Kostlan inner product} shows that the Kostlan inner product is invariant under this action, hence so is the distribution of $s_d$. In other terms, the random process $(s_d(x))_{x \in M}$ is stationary.
\end{exa}

\begin{exa}
\label{ex not Kostlan}
Let us give another example showing that the complex Fubini--Study model goes beyond Kostlan polynomials. Let $\X$ be a Riemann surface embedded in $\C P^n$. We assume that $\X$ is real, in the sense that it is stable under complex conjugation in $\C P^n$. Note that, in our setting, we can always realize $\X$ in such a way for some $n$ large enough, by Kodaira's Embedding Theorem. Then $M = \X \cap \R P^n$ is the disjoint union of at most $\mathfrak{g}$+1 smooth circles in $\R P^n$, where $\mathfrak{g}$ is the genus of $\X$. Figure~\ref{fig Real Riemann Surface} below shows two examples of this geometric situation.

\begin{figure}[ht]
\hspace{0.05\textwidth}
\subfloat[$M$ is a circle even though $\X$ has genus $2$. \label{fig a}]{\includegraphics{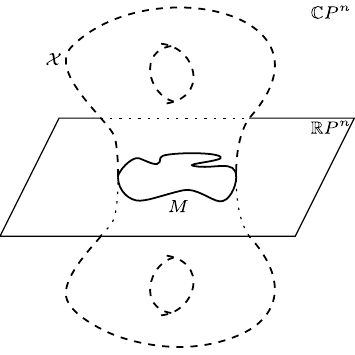}}
\hfill
\subfloat[$\X$ is connected but $M$ has three connected components $M_1$, $M_2$ and $M_3$. \label{fig b}]{\includegraphics{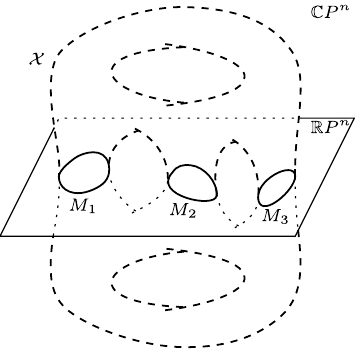}}
\hspace{0.05\textwidth}
\caption{Two examples of a real Riemann surface $\X$ embedded in $\C P^n$ and its real locus $M$.}
\label{fig Real Riemann Surface}
\end{figure}

We choose $\E$ to be trivial and $\L$ to be the restriction to $\X$ of the hyperplane line bundle $\mathcal{O}(1) \to \C P^n$. Then, the elements of $\R\H$ are restrictions to $\X$ of homogeneous polynomials of degree $d$ in $n+1$ variables with real coefficients. This space is equipped with the inner product~\eqref{eq def L2 inner product}. In Equation~\eqref{eq def L2 inner product}, the Kähler form $\omega$ is the restriction $\X$ of the Fubini--Study form on $\C P^n$. However, the domain of integration is $\X$, so that~\eqref{eq def L2 inner product} is not the analogue in $n+1$ variables of the Kostlan inner product of Example~\ref{ex Kostlan inner product}.
\begin{itemize}
\item On Figure~\ref{fig a}, the curve $M$ is connected hence diffeomorphic to $\R P^1$. However, $\E \otimes \L^d \to \X$ is not an avatar of $\mathcal{O}(d) \to \C P^1$ for some $d \geq 1$, since $\X \not\simeq \C P^1$. In particular, the previous construction gives a random homogeneous polynomial $s_d$ on $M \simeq \R P^1$, which is Gaussian and centered but is not a Kostlan polynomial of some degree.

Unlike what happens in Example~\ref{ex Kostlan polynomials}, rotations in $M$ are not obtained as restriction of isometries of $\X$ (generically, the only isometries of $\X$ are the identity and the complex conjugation). Thus, there is no reason why the process $(s_d(x))_{x \in M}$ should be stationary.

\item On Figure~\ref{fig b}, the real locus $M$ had several connected components while $\X$ is connected. Since the inner product~\eqref{eq def L2 inner product} is defined by integrating on the whole complex locus~$\X$, the values of $s_d$ in different connected components of $M$ are a priori correlated.
\end{itemize}
\end{exa}

\begin{lem}[Boundedness of linear statistics]
\label{lem L inf random variables}
Let $d \geq 1$ and $s_d \sim \mathcal{N}(0,\Id)$ in $\R\H$. For all $\phi \in \mathcal{C}^0(M)$, the random variable $\prsc{\nu_d}{\phi}$ is bounded. In particular, it admits finite moments of any order.
\end{lem}

\begin{proof}
We have $\norm{\prsc{\nu_d}{\phi}} = \norm{\sum_{x \in Z_d}\phi(x)} \leq \Norm{\phi}_\infty \card(Z_d)$. Hence it is enough to prove that $\card(Z_d)$ is a bounded random variable.

The cardinality of $Z_d$ is the number of zeros of $s_d$ in $M$, which is smaller than the number of zeros of $s_d$ in $\X$. Now, almost surely, $s_d \neq 0$, and the complex zero set of $s_d$ defines a divisor which is Poincaré-dual to the first Chern class of $\E \otimes \L^d$ (see~\cite[pp.~136 and 141]{GH1994}). Hence, almost surely:
\begin{equation*}
\card(Z_d) \leq \card \left\{ x \in \X \mvert s_d(x)=0 \right\} \leq \int_\X c_1(\E \otimes \L^d)=d\cdot\deg(\L)+\deg(\E).\qedhere
\end{equation*}
\end{proof}

\begin{rem}
\label{rem roots}
In the case of polynomials, the proof is clearer: the number of real roots of a non-zero polynomial of degree $d$ is bounded by the number of its complex roots, which is at most $d$.
\end{rem}


\subsection{Kac-Rice formulas and density functions}
\label{subsec Kac-Rice formulas and density functions}

In this section, we recall some important facts about Kac-Rice formulas. These formulas are classical tools in the study of moments of local quantities such as the cardinality, or more generally the volume, of the zero set of a smooth Gaussian process. Classical references for this material are~\cite{AT2007,AW2009}. With a more geometric point of view, the following formulas were proved and used in~\cite{Anc2021,GW2016,LP2019}, see also~\cite{Let2016}. In the same spirit, Lerario and Stecconi derived a Kac--Rice formula for sections of fiber bundles, see~\cite[Theorem~23.6]{LS2019}.

\begin{rem}
In some of the papers we refer to in this section, the line bundle $\E$ is taken to be trivial. That is the authors considers random real sections of $\L^d$ instead of $\E \otimes \L^d$. As we already explained (see Remark~\ref{rem line bundle E}), the proofs of the results we cite below rely on asymptotics for the Bergman kernel of $\E \otimes \L^d$, as $d \to +\infty$. These asymptotics do not depend on $\E$, see~\cite[Theorem~4.2.1]{MM2007} and~\cite[Theorem~1]{MM2015}. Hence the results established in the case of a trivial line bundle $\E$ can be extended to the case of a general $\E$ without modifying the proofs.
\end{rem}

\begin{dfn}[Jacobian]
\label{def jacobian}
Let $V$ and $V'$ be two Euclidean spaces of dimension $N$ and $N'$ respectively. Let $L:V \to V'$ be a linear map and let $L^*$ denote its adjoint. We denote by $\odet{L} = \det\left(LL^*\right)^\frac{1}{2}$ the \emph{Jacobian} of $L$.
\end{dfn}

\begin{rem}
If $N' \leq N$, an equivalent definition of the Jacobian is the following: $\odet{L}$ is the absolute value of the determinant of the matrix of $L$, restricted to $\ker(L)^\perp$, in orthonormal bases. Note that, in any case, $\odet{L} > 0$ if and only if $L$ is surjective.
\end{rem}
Let us now consider the geometric setting of Section~\ref{subsec geometric setting}.

\begin{dfn}[Evaluation map]
\label{def ev dx}
Let $k \in \N^*$, for all $d \in \N$, for all $x=(x_1,\dots,x_k)\in M^k$, we denote by:\begin{equation*}
\ev^d_x:\R\H \longrightarrow \bigoplus_{i=1}^k \R(\E \otimes \L^d)_{x_i}
\end{equation*}
the \emph{evaluation map} at $x$, defined by $\ev^d_x:s\mapsto (s(x_1),\dots,s(x_k))$.
\end{dfn}

\begin{lem}[{\cite[Proposition~2.11]{Anc2021}}]
\label{lem surjectivity ev dx}
Let $k \in \N^*$, there exists $d_k \in \N$ such that, for all $d \geq d_k$, for all $x \in M^k \setminus \Delta_k$, the evaluation map $\ev^d_x$ is surjective (i.e.~$\odet{\ev_x^d}>0$).
\end{lem}

\begin{rem}
\label{rem dk to infinity}
The dimension $N_d$ of $\R\H$ does not depend on $k$, while the dimension of the target space of $\ev^d_x$ equals $k$. In particular, for any $d$, the linear map $\ev_x^d$ can only be surjective if $k \leq N_d$. This shows that $d_k \xrightarrow[k \to +\infty]{} +\infty$. Moreover, $(d_k)_{k \geq 1}$ is non-decreasing.
\end{rem}

\begin{rem}
\label{rem non degeneracy}
If $s_d \sim \mathcal{N}(0,\Id)$ in $\R\H$, then $\ev^d_x(s_d) = (s_d(x_1),\dots,s_d(x_k))$ is a centered Gaussian variable in $\bigoplus_{i=1}^k \R(\E \otimes \L^d)_{x_i}$ whose variance operator is $\ev^d_x (\ev^d_x)^*$. If $d \geq d_k$, then for all $x \in M^k \setminus \Delta_k$, this variance operator is positive and $\ev^d_x(s_d)$ is non-degenerate.
\end{rem}

\begin{dfn}[Kac--Rice density]
\label{def Rkd}
Let $k \in \N^*$, for all $d \geq d_k$ (cf.~Lemma~\ref{lem surjectivity ev dx}), we define the \emph{density function} $\mathcal{R}^k_d:M^k \setminus \Delta_k \to \R$ as follows:
\begin{equation*}
\forall x=(x_1,\dots,x_k) \in M^k \setminus \Delta_k, \qquad \mathcal{R}^k_d(x) = (2\pi)^{-\frac{k}{2}}\frac{\espcond{\prod_{i=1}^k \Norm{\nabla_{x_i} s_d}}{\ev_x^d(s_d)=0}}{\odet{\ev^d_x}}.
\end{equation*}
Here, $\nabla$ is any connection on $\E \otimes \L^d$ and $\espcond{\prod_{i=1}^k \Norm{\nabla_{x_i} s_d}}{\ev_x^d(s_d)=0}$ stands for the conditional expectation of $\prod_{i=1}^k \Norm{\nabla_{x_i} s_d}$ given that $\ev_x^d(s_d)=0$.
\end{dfn}

\begin{rem}
\label{rem choice of nabla}
Recall that if $s$ is a section and $s(x)=0$, then the derivative $\nabla_x s$ does not depend on the choice of the connection $\nabla$. In the conditional expectation appearing at the numerator in Definition~\ref{def Rkd}, we only consider derivatives of $s_d$ at the points $x_1,\dots,x_k$, under the condition that $s_d$ vanishes at these points. Hence $\mathcal{R}^k_d$ does not depend on the choice of $\nabla$.
\end{rem}

We can now state the Kac--Rice formula we are interested in, see~\cite[Propositions~2.5 and~2.9]{Anc2021}. See also~\cite[Theorem~5.5]{LP2019} in the case $k=2$. Recall that $\tilde{\nu}_d^k$ was defined by Definition~\ref{def nu A} and is the counting measure of the random set $\left(s_d^{-1}(0) \cap M\right)^k \setminus \Delta_k$. 

\begin{prop}[Kac--Rice formula]
\label{prop Kac Rice}
Let $k \in \N^*$ and let $d \geq d_k$. Let $s_d \sim \mathcal{N}(0,\Id)$ in $\R\H$, for any $\phi \in \mathcal{C}^0(M^k)$, we have:
\begin{equation*}
\esp{\prsc{\tilde{\nu}_d^k}{\phi}} = \int_{x \in M^k} \phi(x) \mathcal{R}^k_d(x) \rmes{M}^k.
\end{equation*}
where $\mathcal{R}^k_d$ is the density function defined by Definition~\ref{def Rkd}.
\end{prop}

Let $k \geq 2$ and $d \in \N$. If $x \in \Delta_k$, then the evaluation map $\ev_x^d$ can not be surjective. Hence, the continuous map $x \mapsto \odet{\ev^d_x}$ from $M^k$ to $\R$ vanishes on $\Delta_k$, and one would expect $\mathcal{R}^k_d$ to be singular along the diagonal. Yet, Ancona showed that one can extend continuously $\mathcal{R}^k_d$ to the whole of $M^k$ and that the extension vanishes on $\Delta_k$, see~\cite[Theorem~4.7]{Anc2021}. Moreover, he showed that $d^{-\frac{k}{2}}\mathcal{R}^k_d$ is uniformly bounded on $M^k \setminus \Delta_k$ as $d \to +\infty$. We will use this last fact repeatedly in the proof of Theorem~\ref{thm main}, see Section~\ref{sec asymptotics of the central moments} below.

\begin{prop}[{\cite[Theorem~3.1]{Anc2021}}]
\label{prop bounded density}
For any $k\in \N^*$, there exists a constant $C_k >0$ such that, for all $d$ large enough, for all $x \in M^k \setminus \Delta_k$, we have $d^{-\frac{k}{2}}\mathcal{R}^k_d(x)\leq C_k$.
\end{prop}

Let $d \geq 1$ and let $s_d \in \R\H$ be a standard Gaussian. A fundamental idea in our problem is that the values (and more generally the $k$-jets) of~$s_d$ at two points $x$ and $y \in M$ are ``quasi-independent'' if $x$ and $y$ are far from one another, at scale~$d^{-\frac{1}{2}}$. More precisely, $(s_d(x))_{x \in M}$ defines a Gaussian process with values in $\R(\E \otimes \L^d)$ whose correlation kernel is the Bergman kernel $E_d$ of $\E \otimes \L^d$. Recall that $E_d$ is the integral kernel of the orthogonal projection from the space of square integrable sections of $\E \otimes \L^d$ onto $\H$, for the inner product defined by Equation~\eqref{eq def L2 inner product}. In particular, for any $x,y \in \X$, we have $E_d(x,y) \in (\E \otimes \L^d)_x \otimes (\E \otimes \L^d)^*_y$. For our purpose, it is more convenient to consider the normalized Bergman kernel
\begin{equation*}
e_d:(x,y) \mapsto E_d(x,x)^{-\frac{1}{2}}E_d(x,y)E_d(y,y)^{-\frac{1}{2}}.
\end{equation*}
For any $x \in \X$, the map $E_d(x,x)$ is an endomorphism of the $1$-dimensional space $(\E \otimes \L^d)_x$, hence can be seen as a scalar. Note that $e_d$ is the correlation kernel of the normalized process $\left(E_d(x,x)^{-\frac{1}{2}}s_d(x)\right)_{x \in M}$, which has unit variance and the same zero set as $(s_d(x))_{x \in M}$.

The normalized Bergman kernel $e_d$ admits a universal local scaling limit at scale~$d^{-\frac{1}{2}}$ around any point of~$\X$ (cf.~\cite[Theorem~4.2.1]{MM2007}). Moreover, the correlations are exponentially decreasing at scale~$d^{-\frac{1}{2}}$, in the sense that there exists $C>0$ such that $\Norm{e_d(x,y)} = O\!\left(\exp\left(-C\sqrt{d}\rho_g(x,y)\right)\right)$ as $d \to +\infty$, uniformly in $(x,y)$, where $\rho_g$ is the geodesic distance. Similar estimates hold for the derivatives of $e_d$ with the same constant $C$. We refer to~\cite[Theorem~1]{MM2015} for a precise statement. These facts were extensively used in~\cite{Anc2021,Let2019,LP2019} and we refer to these papers for a more detailed discussion of how estimates for the Bergman kernel are used in the context of random real algebraic geometry. An important consequence of these estimates that we use in the present paper is Proposition~\ref{prop multiplicativity Rk} below.

\begin{dfn}
\label{def bp}
Let $p \in \N$, we denote by $b_p = \frac{1}{C}\left(1+\frac{p}{4}\right)$, where $C>0$ is the same as above. That is $C$ is the constant appearing in the exponential in~\cite[Theorem~1, Equation~(1.3)]{MM2015}.
\end{dfn}

\begin{prop}[{\cite[Proposition~3.2]{Anc2021}}]
\label{prop multiplicativity Rk}
Let $p \geq 2$. Recall that $\rho_g$ denotes the geodesic distance and that $b_p$ is defined by Definition~\ref{def bp}. The following holds uniformly for all $k \in \{2,\dots,p\}$, for all $A$ and $B \subset \{1,\dots,k\}$ disjoint such that $A \sqcup B = \{1,\dots,k\}$, for all $x \in M^k \setminus \Delta_k$ such that for all $a \in A$ and $b \in B$ we have $\rho_g(x_a,x_b) \geq b_p \frac{\ln d}{\sqrt{d}}$:
\begin{equation*}
\mathcal{R}^k_d(x) = \mathcal{R}^{\norm{A}}_d(\underline{x}_A)\mathcal{R}^{\norm{B}}_d(\underline{x}_B) + O(d^{\frac{k}{2}-\frac{p}{4}-1}).
\end{equation*}
Here we used the notations defined in Section~\ref{subsec partitions products and diagonal inclusions} (see Notation~\ref{ntn product indexed by A}), and by convention $\mathcal{R}_d^0=1$ for all $d$.
\end{prop}

\begin{proof}
Proposition~\ref{prop multiplicativity Rk} is the same as \cite[Proposition~3.2]{Anc2021} but for two small points. We refer to~\cite{Anc2021} for the core of the proof of this proposition. Here, let us just explain what the differences are between Proposition~\ref{prop multiplicativity Rk} and \cite[Proposition~3.2]{Anc2021}, and how these differences affect the proof.

\begin{itemize}
\item Here we consider a line bundle of the form $\E \otimes \L^d$, while in \cite{Anc2021} the author only considers~$\L^d$, which corresponds to the case where $\E$ is trivial. In the proof of~\cite[Proposition~3.2]{Anc2021}, the geometry of line bundle $\L^d$ only appears through the leading term in the asymptotics of its Bergman kernel, as $d \to +\infty$. More precisely, the necessary estimates are those of~\cite[Theorem~4.2.1]{MM2007} and~\cite[Theorem~1]{MM2015}. One can check that, if we replace $\L^d$ by $\E \otimes \L^d$ in the estimates of~\cite{MM2007,MM2015}, we obtain estimates of the same form for the Bergman kernel of $\E \otimes \L^d$. Thus, the addition of the line bundle $\E$ in the problem does not affect the proof, except in a typographical way.

\item In the statement of~\cite[Proposition~3.2]{Anc2021}, the constant $b_p$ is replaced by $\frac{1}{C}=b_0$ and the error term is replaced by $O(d^{\frac{k}{2}-1})$. Let us explain how using $b_p$ instead of $b_0$ yields an error term of the form $O(d^{\frac{k}{2}-\frac{p}{4}-1})$. The error term in Proposition~\ref{prop multiplicativity Rk} and~\cite[Proposition~3.2]{Anc2021} comes from the off-diagonal decay estimate of~\cite[Theorem~1]{MM2015}. More precisely, in both cases, the term:
\begin{equation*}
\norm{\mathcal{R}^k_d(x) - \mathcal{R}^{\norm{A}}_d(\underline{x}_A)\mathcal{R}^{\norm{B}}_d(\underline{x}_B)}
\end{equation*}
is bounded from above by some constant (that does not depend on $x$ and $d$) times the $\mathcal{C}^k$-norm of the normalized Bergman kernel $e_d$ of $\E \otimes \L^d$ restricted to $\left\{(x,y) \in M^2 \mid \rho_g(x,y) \geq L \right\}$, where
\begin{equation*}
L = \min_{(a,b) \in A \times B} \rho_g(x_a,x_b).
\end{equation*}
By~\cite[Theorem~1]{MM2015}, this term is $O\!\left(d^\frac{k}{2}e^{-C\sqrt{d}L}\right)$. In the setting of~\cite[Proposition~3.2]{Anc2021}, we have $L \geq \frac{1}{C} \frac{\ln d}{\sqrt{d}}$, which gives an error term of the form $O(d^{\frac{k}{2}-1})$. In Proposition~\ref{prop multiplicativity Rk}, we have $L \geq \frac{1}{C}(1+\frac{p}{4})\frac{\ln d}{\sqrt{d}}$, so that the error term is indeed $O(d^{\frac{k}{2}-\frac{p}{4}-1})$.\qedhere
\end{itemize}
\end{proof}

\begin{rems}
\label{rem Michele's paper}
We conclude this section with some comments for the reader who might be interested in the proof of~\cite[Proposition~3.2]{Anc2021}.
\begin{itemize}
\item In~\cite{Anc2021}, estimates for the Bergman kernel (i.e. the correlation function of the random process under study) are obtained using peak sections. This alternative method yields the necessary estimates without having to use the results of~\cite{MM2007,MM2015}.

\item In~\cite{Anc2021}, the proof of Proposition~3.2 is written for $k=2$ and $\norm{A}=1=\norm{B}$, for clarity of exposition. The extension to $k >2$ is non-trivial, and the proof for $k >2$ requires in fact the full power of the techniques developed in \cite[Section~4]{Anc2021}. More recently, the authors developed similar techniques for smooth stationary Gaussian processes in $\R$, see~\cite[Theorem~1.14]{AL2021}.
\end{itemize}
\end{rems}


\section{Asymptotics of the central moments}
\label{sec asymptotics of the central moments}

The goal of this section is to prove Theorem~\ref{thm main}. In Section~\ref{subsec an integral expression of the central moments} we derive an integral expression of the central moments we want to estimate, see Lemma~\ref{lem integral expression} below. Then, in Section~\ref{subsec cutting MA into pieces}, we define a decomposition of the manifolds $M^A$, where $M$ is as in Section~\ref{subsec geometric setting} and $A$ is a finite set. In Sections~\ref{subsec an upper bound on the contribution of each piece}, \ref{subsec contribution of the partitions with an isolated point} and~\ref{subsec contribution of the partitions into pairs}, we compute the contributions of the various pieces of the decomposition defined in Section~\ref{subsec cutting MA into pieces} to the asymptotics of the integrals appearing in Lemma~\ref{lem integral expression}. Finally, we conclude the proof of Theorem~\ref{thm main} in Section~\ref{subsec proof of the main theorem}.


\subsection{An integral expression of the central moments}
\label{subsec an integral expression of the central moments}

The purpose of this section is to derive a tractable integral expression for the central moments of the form $m_p(\nu_d)(\phi_1,\dots,\phi_p)$, defined by Definition~\ref{def m p nu d} and appearing in Theorem~\ref{thm main}. This is done in Lemma~\ref{lem integral expression} below. Before stating and proving this lemma, we introduce several concepts and notations that will be useful to deal with the combinatorics of this section and the following ones.

Recall that we already defined the set $\pa_A$ of partitions of a finite set $A$ (see Definition~\ref{def partitions}). The next concept we need to introduce is that of induced partition.

\begin{dfn}[Induced partition]
\label{def induced partition}
Let $B$ be a finite set, let $A \subset B$ and let $\I \in \pa_B$. The partition $\I$ defines a partition $\I_A = \left\{ I \cap A \mvert I \in \I, I \cap A \neq \emptyset \right\}$ of $A$, which is called the partition of $A$ \emph{induced} by $\I$.
\end{dfn}

In this paper, we will only consider the case where $\I \in \pa_B$ and $A$ is the reunion of some of the elements of $\I$, that is: for any $I \in \I$, either $I \subset A$ or $I \cap A = \emptyset$. This condition is equivalent to $\I_A \subset \I$, and in this case we have $\I_A = \{I \in \I \mid I \subset A\}$.

Another useful notion is the following order relation on partitions.

\begin{dfn}[Order on partitions]
\label{def order on partitions}
Let $A$ be a finite set and let $\I,\J \in \pa_A$. We denote by $\J \leq \I$ and we say that $\J$ is \emph{finer} than $\I$ if $\J$ is obtained by subdividing the elements of $\I$. That is, for any $J \in \J$ there exists $I \in \I$ such that $J \subset I$. If $\J \leq \I$ and $\J \neq \I$, we say that $\J$ is \emph{strictly finer} than $\I$, which we denote by $\J < \I$.
\end{dfn}

One can check the following facts. Given a finite set $A$, the relation $\leq$ defines a partial order on $\pa_A$ such that $\I \mapsto \norm{\I}$ is decreasing. The partially ordered set $\left(\pa_A,\leq \right)$ admits a unique maximum equal to $\{A\}$. It also admits a unique minimum equal to $\left\{ \{a\} \mvert a \in A \right\}$, that we denote by $\I_0(A)$ as in Remark~\ref{rem diagonal inclusions}. If $A$ is of the form $\{1,\dots,k\}$, we use the simpler notations $\I_0(k)=\I_0(A)$. Finally, note that $\J \leq \I$ if and only if $\J = \bigsqcup_{I \in \I} \J_I$.

The last thing we need to define is the notion of subset adapted to a partition. A subset $A$ of a finite set $B$ is adapted to the partition $\I \in \pa_B$ if $B \setminus A$ is a union of singletons of $\I$. In other words, $A$ is adapted to $\I$ if and only if $\I = \I_A \sqcup \{\{i\} \mid i \in B \setminus A\}$, or equivalently if and only if $\I \leq \{A\} \sqcup \I_0(B \setminus A)$. Here is the formal definition.

\begin{dfn}[Subset adapted to a partition]
\label{def subset adapted to I}
Let $B$ be a finite set and let $\I \in \pa_B$, we denote~by:
\begin{equation*}
\mathcal{S}_\I = \left\{A \subset B \mvert \forall I \in \I, \text{ if } \norm{I} \geq 2, \text{ then } I \subset A \right\}.
\end{equation*}
A subset $A \in \mathcal{S}_\I$ is said to be \emph{adapted} to $\I$.
\end{dfn}

The set $\mathcal{S}_\I$ will appear as an index set in the integral expression derived in Lemma~\ref{lem integral expression} below. The only thing we need to know about it is the following natural result.

\begin{lem}
\label{lem subsets adapted to I}
Let $B$ be any finite set, then the map $(A,\I) \mapsto (A,\I_A)$ defines a bijection from $\{(A,\I) \mid \I\in \pa_B, A \in \mathcal{S}_\I \}$ to $\{ (A,\J) \mid A \subset B, \J \in \pa_A \}$.
\end{lem}

\begin{proof}
The inverse map is given by $(A,\J) \mapsto (A, \J \sqcup \{\{i\} \mid i \in B \setminus A\})$.
\end{proof}

We can now derive the integral expression of the central moments we are looking for. We will make use of the various notations introduced in Section~\ref{subsec partitions products and diagonal inclusions}.

\begin{lem}
\label{lem integral expression}
Let $p \geq 2$ and let $\phi_1,\dots,\phi_p \in \mathcal{C}^0(M)$. Let $d_p \in \N$ be given by Lemma~\ref{lem surjectivity ev dx}, for all $d \geq d_p$ we have:
\begin{equation*}
m_p(\nu_d)(\phi_1,\dots,\phi_p)= \sum_{\I \in \pa_p} \int_{\underline{x}_\I \in M^\I} \iota_\I^*\phi(\underline{x}_\I)\, \D^\I_d(\underline{x}_\I) \rmes{M}^{\norm{\I}},
\end{equation*}
where, for any finite set $B$, any $\I \in \pa_B$ and any $\underline{x}_\I = (x_I)_{I \in \I} \in M^\I$,
\begin{equation*}
\D^\I_d(\underline{x}_\I) = \sum_{A \in \mathcal{S}_\I} (-1)^{\norm{B} - \norm{A}} \mathcal{R}_d^{\norm{\I_A}}(\underline{x}_{\I_A}) \prod_{i \notin A} \mathcal{R}_d^1(x_{\{i\}}).
\end{equation*}
Here, $\mathcal{S}_\I$ is the one we defined in Definition~\ref{def subset adapted to I}.
\end{lem}

\begin{proof}
By definition, $m_p(\nu_d)(\phi_1,\dots,\phi_p)$ is the expectation of a linear combination of integrals over sets of the form $Z_d^k$, with $k \in \{1,\dots,p\}$. We subdivide each $Z_d^k$ according to the various diagonals in $M^k$. The corresponding splitting of the counting measure $\nu_d^k$ of $Z_d^k$ was computed in Lemma~\ref{lem decomposition nu ks}. We can then apply the Kac--Rice formula to each term. The result follows by carefully reordering the terms.

We start from the definition of $m_p(\nu_d)$, see Definition~\ref{def m p nu d}. Developing the product, we get:
\begin{align*}
m_p(\nu_d)(\phi_1,\dots,\phi_p) &= \sum_{A \subset \{1,\dots,p\}} (-1)^{p-\norm{A}} \esp{\prod_{i \in A} \prsc{\nu_d}{\phi_i}} \prod_{i \notin A}\esp{\prsc{\nu_d}{\phi_i}}\\
&= \sum_{A \subset \{1,\dots,p\}} (-1)^{p-\norm{A}} \esp{\prsc{\nu_d^A}{\phi_A}} \prod_{i \notin A}\esp{\prsc{\nu_d}{\phi_i}}\\
&= \sum_{A \subset \{1,\dots,p\}} \sum_{\I \in \pa_A} (-1)^{p-\norm{A}} \esp{\prsc{\tilde{\nu}_d^\I}{\iota_\I^*\phi_A}} \prod_{i \notin A}\esp{\prsc{\nu_d}{\phi_i}},
\end{align*}
where the last equality comes from Lemma~\ref{lem decomposition nu ks}. By the Kac--Rice formulas of Proposition~\ref{prop Kac Rice}, this equals:
\begin{equation*}
\sum_{A \subset \{1,\dots,p\}} \sum_{\I \in \pa_A} (-1)^{p-\norm{A}} \left(\int_{M^\I} (\iota_\I^*\phi_A) \mathcal{R}_d^{\norm{\I}} \rmes{M}^{\norm{\I}}\right) \prod_{i \notin A} \left(\int_M \phi_i \mathcal{R}_d^1 \rmes{M}\right).
\end{equation*}
By Lemma~\ref{lem subsets adapted to I}, we can exchange the two sums and obtain the following:
\begin{multline*}
\sum_{\I \in \pa_p} \sum_{A \in \mathcal{S}_\I} (-1)^{p-\norm{A}} \left(\int_{M^{\I_A}} (\iota_{\I_A}^*\phi_A) \mathcal{R}_d^{\norm{\I_A}} \rmes{M}^{\norm{\I_A}}\right) \prod_{i \notin A} \left(\int_M \phi_i \mathcal{R}_d^1 \rmes{M}\right)\\
\begin{aligned}
&= \sum_{\I \in \pa_p} \sum_{A \in \mathcal{S}_\I} (-1)^{p-\norm{A}} \int_{\underline{x}_\I \in M^\I} \iota_\I^*\phi(\underline{x}_\I) \mathcal{R}_d^{\norm{\I_A}}(\underline{x}_{\I_A}) \prod_{i \notin A} \mathcal{R}_d^1(x_{\{i\}}) \rmes{M}^{\norm{\I}}\\
&= \sum_{\I \in \pa_p} \int_{\underline{x}_\I \in M^\I} \iota_\I^*\phi(\underline{x}_\I)\,\D_d^\I(\underline{x}_\I) \rmes{M}^{\norm{\I}},
\end{aligned}
\end{multline*}
which concludes the proof.
\end{proof}

Before going further, let us try to give some insight into what is going to happen. Our strategy is to compute the large $d$ asymptotics of the terms of the form
\begin{equation}
\label{eq integral MI}
\int_{\underline{x}_\I \in M^\I} \iota_\I^*\phi(\underline{x}_\I) \D^\I_d(\underline{x}_\I) \rmes{M}^{\norm{\I}}
\end{equation}
appearing in Lemma~\ref{lem integral expression}. In the case where $p$ is even, which is a bit simpler to describe, many of these terms will contribute a leading term of order $d^\frac{p}{4}$, and the others will only contribute a smaller order term. Terms of the form~\eqref{eq integral MI} can all be dealt with in the same way, hence we will not make any formal difference between them in the course of the proof. Note however, that the term indexed by $\I_0(p)$ is simpler to understand, as shown in Example~\ref{ex term I0} below. This term is the one we considered in our sketch of proof in Section~\ref{sec introduction}. At each step in the following, we advise the reader to first understand what happens for this term, before looking at the general case.

\begin{exa}[Term indexed by $\I_0(p)$]
\label{ex term I0}
Recall that $\I_0(p) = \left\{\{i\} \mvert 1 \leq i \leq p \right\}$. Identifying $\I_0(p)$ with $\{1,\dots,p\}$ through the canonical bijection  $i \mapsto \{i\}$, the map $\iota_{\I_0(p)}:M^{\I_0(p)} \to M^p$ is the identity of $M^p$, see Definition~\ref{def diagonal inclusions}. Hence, the term indexed by $\I_0(p)$ in Lemma~\ref{lem integral expression} is:
\begin{equation*}
\int_{(x_1,\dots,x_p) \in M^p} \left(\prod_{i=1}^p\phi_i(x_i)\right) \D_d^{\I_0(p)}(x_1,\dots,x_p) \rmes{M}^p.
\end{equation*}

Then, $\mathcal{S}_{\I_0(p)}$ is just the set of all subsets of $\{1,\dots,p\}$, see Definition~\ref{def subset adapted to I}. Furthermore, for any $A \subset \{1,\dots,p\}$ the induced partition $\I_0(p)_A$ equals $\{\{i\} \mid i \in A\}$, hence is canonically in bijection with~$A$. Finally, we obtain that:
\begin{equation*}
\D_d^{\I_0(p)} : (x_1,\dots,x_p) \longmapsto \sum_{A \subset \{1,\dots,p\}} (-1)^{p-\norm{A}} \mathcal{R}_d^{\norm{A}}(\underline{x}_A) \prod_{i \notin A} \mathcal{R}^1_d(x_i).
\end{equation*}
\end{exa}


\subsection{Cutting \texorpdfstring{$M^A$}{} into pieces}
\label{subsec cutting MA into pieces}

In this section, we define a way to cut the Cartesian product $M^A$ into disjoint pieces, for any finite subset $A$. The upshot is to use this decomposition of $M^A$ with $A = \I \in \pa_p$, in order to obtain the large $d$ asymptotics of integrals of the form~\eqref{eq integral MI}. For this to work, we need to define a splitting of $M^A$ that depends on the parameters $p \in \N$ and $d \geq 1$.

\begin{dfn}[Clustering graph]
\label{def Gdp}
Let $A$ be a finite set and let $p \in \N$. For all $d \geq 1$, for all $x =(x_a)_{a \in A} \in M^A$, we define a graph $G^p_d(x)$ as follows.
\begin{itemize}
\item The vertices of $G^p_d(x)$ are the elements of $A$.
\item Two vertices $a$ and $b \in A$ are joined by an edge of $G^p_d(x)$ if and only if $\rho_g(x_a,x_b) \leq b_p \frac{\ln d}{\sqrt{d}}$ and $a \neq b$, where $\rho_g$ is the geodesic distance in $M$ and $b_p$ is the constant defined by Definition~\ref{def bp}.
\end{itemize}
\end{dfn}

The clustering graph $G^p_d(x)$ contains more information than we need. What we are actually interested in is the partition of $A$ given by the connected components of $G^p_d(x)$. This partition, defined below, encodes how the components of $x$ are clustered in $M$ at scale $b_p\frac{\ln d}{\sqrt{d}}$. An example of this construction is given on Figure~\ref{fig clusters} below, for $A = \{1,\dots,6\}$ and $M$ a circle.

\begin{dfn}[Clustering partition]
\label{def Idp}
Let $A$ be a finite set and let $p \in \N$. For all $d \geq 1$, for all $x \in M^A$, we denote by $\I_d^p(x)$ the partition of $A$ defined by the connected components of~$G_d^p(x)$. That is $a$ and $b \in A$ belong to the same element of $\I_d^p(x)$ if and only if they are in the same connected component of $G_d^p(x)$.
\end{dfn}

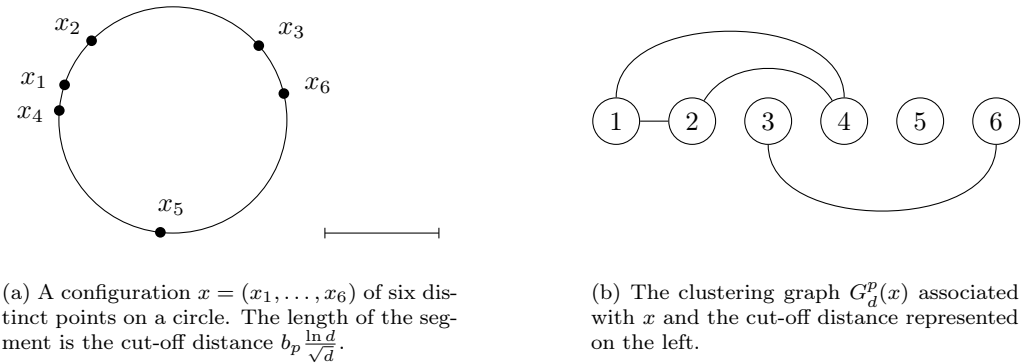
\begin{figure}[ht]
\hspace{0.05\textwidth}
\subfloat[A configuration $x = (x_1,\dots,x_6)$ of six distinct points on a circle. The length of the segment is the cut-off distance $b_p \frac{\ln d}{\sqrt{d}}$. \label{fig config}]{\begin{tikzpicture}[x=0.5cm,y=0.5cm]
\clip(-4.5,-4) rectangle (7.5,4);
\draw(0,0) circle (1.5cm);
\draw (4,-3)-- (7,-3);
\fill [color=black] (-2.85,0.93) circle (2pt);
\draw[color=black] (-3.65,1.06) node {$x_1$};
\fill [color=black] (-2.14,2.1) circle (2pt);
\draw[color=black] (-2.76,2.48) node {$x_2$};
\fill [color=black] (2.26,1.97) circle (2pt);
\draw[color=black] (3.13,2.37) node {$x_3$};
\fill [color=black] (-2.99,0.25) circle (2pt);
\draw[color=black] (-3.75,0.15) node {$x_4$};
\fill [color=black] (-0.33,-2.98) circle (2pt);
\draw[color=black] (-0.04,-2.31) node {$x_5$};
\fill [color=black] (2.92,0.7) circle (2pt);
\draw[color=black] (3.82,0.92) node {$x_6$};
\draw [color=black,shift={(4,-3)}] (0,-2pt)--(0,2pt);
\draw [color=black,shift={(7,-3)}] (0,-2pt)--(0,2pt);
\end{tikzpicture}}
\hfill
\subfloat[The clustering graph $G^p_d(x)$ associated with $x$ and the cut-off distance represented on the left. \label{fig graph}]{\begin{tikzpicture}
\tikz{
	\node (a) [circle,draw] at (0.3,0.2) {$1$};
	\node (b) [circle,draw] at (1.3,0.2) {$2$};
	\node (c) [circle,draw] at (2.3,0.2) {$3$};
	\node (d) [circle,draw] at (3.3,0.2) {$4$};
	\node (e) [circle,draw] at (4.3,0.2) {$5$};
	\node (f) [circle,draw] at (5.3,0.2) {$6$};
	
	\graph{
		(a) -- (b);
		(a) --[bend left=60] (d);
		(b) --[bend left=45] (d);
		(c) --[bend right=60] (f)
	};
}
\end{tikzpicture}}
\hspace{0.05\textwidth}
\caption{Example of a configuration of six points on a curve $M$ isometric to a circle. The clustering partition of $x = (x_1,\dots,x_6)$ is the partition of $\{1,\dots,6\}$ according to the connected components of the graph~$G^p_d(x)$. Here it equals $\I^p_d(x) = \left\{\{1,2,4\},\{3,6\},\{5\}\right\}$. \label{fig clusters}}
\end{figure}

Given $p \in \N$ and $d \geq 1$, we can now divide the points of $M^A$ according to their clustering partition at scale $b_p\frac{\ln d}{\sqrt{d}}$. This yields a splitting of $M^A$ into pieces indexed by $\pa_A$.

\begin{dfn}
\label{def M AIdp}
Let $A$ be a finite set, let $\I \in \pa_A$, let $p \in \N$ and let $d \geq 1$, we define:
\begin{equation*}
M^{A,p}_{\I,d} = \left\{ x \in M^A \mvert \I^p_d(x) = \I \right\}.
\end{equation*}
\end{dfn}

\begin{rems}
\label{rem M AIdp}
\begin{itemize}
\item For any $p \in \N$ and $d \geq 1$, we have: $M^A = \bigsqcup_{\I \in \pa_A} M^{A,p}_{\I,d}$.
\item Let $\I \in \pa_A$ and let $(x_a)_{a \in A} \in M^{A,p}_{\I,d}$. For all $I \in \I$, for any $a$ and $b \in I$, the geodesic distance from $x_a$ to $x_b$ satisfies: $\rho_g(x_a,x_b) \leq (\norm{I}-1) b_p\frac{\ln d}{\sqrt{d}} \leq \norm{A} b_p\frac{\ln d}{\sqrt{d}}$.
\item Recalling Definitions~\ref{def diagonals} and~\ref{def order on partitions}, if $d >1$ the set $M^{A,p}_{\I,d}$ is a neighborhood of the diagonal $\Delta_{A,\I}$ minus some neighborhood of $\bigsqcup_{\J > \I} \Delta_{A,\J}$. If $d=1$, we have $M^{A,p}_{\I,1}=\Delta_{A,\I}$.
\end{itemize}
\end{rems}

In the course of the proof of Theorem~\ref{thm main}, we will make use of the following two auxiliary results concerning the sets $M^{A,p}_{\I,d}$ defined above.

\begin{lem}
\label{lem M AIdp restricted}
Let $A$ be a finite set, let $\I \in \pa_A$, let $p \in \N$ and let $d \geq 1$. Let $B \subset A$ be a union of elements of $\I$, i.e.~for all $I \in \I$ we have either $I \subset B$ or $I \cap B = \emptyset$. Then, for all $\underline{x}_A \in M^{A,p}_{\I,d}$, we have $\underline{x}_B \in M^{B,p}_{\I_B,d}$.
\end{lem}

\begin{proof}
Let $\underline{x} = (x_a)_{a \in A} \in M^{A,p}_{\I,d}$, recall that $\underline{x}_B = (x_a)_{a \in B} \in M^B$. By Definition~\ref{def M AIdp}, we want to prove that $\I^p_d(\underline{x}_B) = \I_B$, where $\I_B = \{I \in \I \mid I \subset B\}$ is the partition of $B$ induced by $\I$, see Definition~\ref{def induced partition}.

Recalling Definitions~\ref{def Gdp} and~\ref{def Idp}, the partition $\I^p_d(\underline{x}_B)$ is defined by the connected components of $G^p_d(\underline{x}_B)$. The graph $G^p_d(\underline{x}_B)$ is a subgraph of $G^p_d(\underline{x}_A)$, obtained by erasing the vertices indexed by $a \in A \setminus B$ and the edges having at least one endpoint in $A \setminus B$. If $B$ was any subset of $A$, two vertices $a,b \in B$ in the same connected component of $G^p_d(\underline{x}_A)$ could end up in different connected components of $G^p_d(\underline{x}_B)$. In terms of clustering partitions, if $I \in \I_d^p(\underline{x}_A)$, then $I \cap B$ might be a disjoint union of several elements of $\I_d^p(\underline{x}_B)$, neither empty nor equal to $I$. Our hypothesis that $B$ is a union of elements of $\I$ ensures that this does not happen.

To conclude the proof it is enough to show that, for all $I \in \I$ such that $I \subset B$, we have $I \in \I_d^p(\underline{x}_B)$. Let $I \in \I$ such that $I \subset B$. Since $I \in \I = \I_d^p(\underline{x}_A)$, the elements of $I$ are the vertices of one of the connected components of $G_d^p(\underline{x}_A)$, which is equal to $G_d^p(\underline{x}_I)$. This connected component is not modified when we erase the vertices in $A \setminus B$ and the corresponding edges from $G_d^p(\underline{x}_A)$. Thus, $G_d^p(\underline{x}_I)$ is still a connected component in $G_d^p(\underline{x}_B)$, and $I \in \I_d^p(\underline{x}_B)$.
\end{proof}

\begin{lem}
\label{lem M AIdp blurred}
Let $A$ be a finite set, let $\I \in \pa_A$, let $p \in \N$ and let $d \geq 1$, we have the following inclusion between subsets of $M^A$:
\begin{equation*}
M^{A,p}_{\I,d} \subset \prod_{I \in \I} M^{I,p}_{\{I\},d} \subset \bigsqcup_{\J \geq \I} M^{A,p}_{\J,d}.
\end{equation*}
\end{lem}

\begin{proof}
Let $\underline{x}_A= (x_a)_{a \in A} \in M^{A,p}_{\I,d}$. We know that the components of $\underline{x}_A$ are clustered in~$M$ according to $\I$. Informally, if we remember that the points indexed by each $I \in \I$ are close together, but we forget about the distances between points in different clusters, then all we can say is that $\underline{x}_A \in \prod_{I \in \I} M^{I,p}_{\{I\},d}$. More formally, let $I \in \I$, then $I$ is a union of elements of $\I$, and $\I_I = \{I\}$. By Lemma~\ref{lem M AIdp restricted}, we have $\underline{x}_I \in M^{I,p}_{\{I\},d}$ for any $I \in \I$. Hence $\underline{x}_A =(\underline{x}_I)_{I \in \I} \in \prod_{I \in \I} M^{I,p}_{\{I\},d}$. This proves the first inclusion in Lemma~\ref{lem M AIdp blurred}.

Let us now prove the second inclusion. Let $\underline{x}_A \in \prod_{I \in \I} M^{I,p}_{\{I\},d} \subset M^A$, we want to recover its clustering partition $\I_d^p(\underline{x}_A)$. We know that the components of $\underline{x}_A$ indexed by a given $I \in \I$ are close together in $M$, hence belong to the same cluster. However, given $I$ and $J \in \I$, the points indexed by $I$ and those indexed by $J$ might be close enough in $M$ that they belong to the same cluster. Thus, all we can say is that the clusters of $\underline{x}_A$ are indexed by unions of elements of $\I$, that is $\I_d^p(\underline{x}_A) \geq \I$. This shows that $\underline{x}_A \in \bigsqcup_{\J \geq \I} M^{A,p}_{\J,d}$ and concludes the proof.
\end{proof}

We conclude this section by bounding from above the volume of $M^{A,p}_{\I,d}$. What we are interested in here is an asymptotic upper bound as $d \to +\infty$ with $A$, $\I$ and $p$ fixed.

\begin{lem}[Volume bound]
\label{lem volume M AIdp}
Let $A$ be a finite set, let $\I \in \pa_A$ and let $p \in \N$. As $d \to +\infty$, we have:
\begin{equation*}
\vol{M^{A,p}_{\I,d}} = O\!\left(\left(\frac{\ln d}{\sqrt{d}}\right)^{\norm{A}-\norm{\I}}\right).
\end{equation*}
\end{lem}

\begin{proof}
For all $I \in \I$, we fix a preferred element $a_I$ of $I$. Let $x = (x_a)_{a \in A} \in M^{A,p}_{\I,d}$. For all $I \in \I$, for all $a \in I \setminus \{a_I\}$, the point $x_a$ belongs to the geodesic ball of center $x_{a_I}$ and radius $\norm{A} b_p \frac{\ln d}{\sqrt{d}}$. Since $M$ is $1$-dimensional, this ball has volume $2\norm{A} b_p \frac{\ln d}{\sqrt{d}}$. Hence,
\begin{equation*}
\vol{M^{A,p}_{\I,d}} \leq \vol{M}^{\norm{\I}} \prod_{I \in \I} \left(2\norm{A}b_p \frac{\ln d}{\sqrt{d}}\right)^{\norm{I}-1} = O\!\left(\prod_{I \in \I}\left(\frac{\ln d}{\sqrt{d}}\right)^{\norm{I}-1}\right).
\end{equation*}
The result follows from $\sum_{I \in \I} (\norm{I}-1) = \norm{A}-\norm{\I}$.
\end{proof}


\subsection{An upper bound on the contribution of each piece}
\label{subsec an upper bound on the contribution of each piece}

Let $p \geq 2$, given $\I \in \pa_p$ and test-functions $(\phi_i)_{1 \leq i \leq p}$, we want to estimate the integral~\eqref{eq integral MI} as $d \to +\infty$. This is done by splitting $M^\I$ as in Section~\ref{subsec cutting MA into pieces} and estimating the contribution of each piece. One difficulty is that $\I \in \pa_p$, and the pieces of our decomposition of $M^\I$ are indexed by partitions of $\I$, seen as a finite set. Hence we need to consider partitions of partitions.

Let $\J \in \pa_\I$, we are interested in the large $d$ asymptotics of:
\begin{equation}
\label{eq integral MIpJd}
\int_{\underline{x}_\I \in M^{\I,p}_{\J,d}} \iota_\I^*\phi(\underline{x}_\I)\, \D^\I_d(\underline{x}_\I) \rmes{M}^{\norm{\I}}.
\end{equation}
Note that the integral~\eqref{eq integral MI} is the sum over $\J \in \pa_\I$ of the terms of the form~\eqref{eq integral MIpJd}. Hence, the central moment $m_p(\nu_d)(\phi_1,\dots,\phi_p)$ we are interested in is the double sum of these terms over $\I \in \pa_p$ and $\J \in \pa_\I$, see Lemma~\ref{lem integral expression}. As far as the order of magnitude of the integral~\eqref{eq integral MIpJd} is concerned, we can think of the test-function $\iota_\I^*\phi$ as being constant. This order of magnitude will then result from the competition between two things: the order of the typical values of $\D_d^\I$ and the volume of the set $M^{\I,p}_{\J,d}$.

In this section, we compute an asymptotic upper bound as $d \to +\infty$ for the contribution of the integral~\eqref{eq integral MIpJd}. This upper bound shows that, if $\norm{\J} < \frac{p}{2}$, then the integral~\eqref{eq integral MIpJd} only contributes an error term in the estimates of Theorem~\ref{thm main}. This is not surprising since in this case the volume of $M^{\I,p}_{\J,d}$ is comparatively small, see Lemma~\ref{lem volume M AIdp}.

\begin{lem}
\label{lem upper bound}
Let $p \geq 2$ and let $\phi_1,\dots,\phi_p \in \mathcal{C}^0(M)$. Let $\I \in \pa_p$, for all $\J \in \pa_\I$, we have:
\begin{equation*}
\int_{\underline{x}_\I \in M^{\I,p}_{\J,d}} \iota_\I^*\phi(\underline{x}_\I)\, \D^\I_d(\underline{x}_\I) \rmes{M}^{\norm{\I}} = \left(\prod_{i=1}^p \Norm{\phi_i}_\infty\right) O\!\left(d^{\frac{\norm{\J}}{2}} (\ln d)^{\norm{\I}-\norm{\J}}\right),
\end{equation*}
where the error term does not depend on $(\phi_1,\dots,\phi_p)$.
\end{lem}

\begin{proof}
We have $\Norm{\iota_\I^*\phi}_\infty \leq \Norm{\phi}_\infty \leq \prod_{i=1}^p \Norm{\phi_i}_\infty$. Besides, by the definition of $\D_d^\I$ (see Lemma~\ref{lem integral expression}) and Proposition~\ref{prop bounded density}, the density $\D_d^\I$ is bounded uniformly on $M^\I$ by $O(d^\frac{\norm{\I}}{2})$. Integrating this estimate over $M^{\I,p}_{\J,d}$ yields:
\begin{equation*}
\norm{\int_{\underline{x}_\I \in M^{\I,p}_{\J,d}} \iota_\I^*\phi(\underline{x}_\I) \, \D^\I_d(\underline{x}_\I) \rmes{M}^{\norm{\I}}} \leq \left(\prod_{i=1}^p \Norm{\phi_i}_\infty\right) O(d^{\frac{\norm{\I}}{2}}) \vol{M^{\I,p}_{\J,d}}.
\end{equation*}
Then the result follows from the volume estimate of Lemma~\ref{lem volume M AIdp}.
\end{proof}


\subsection{Contribution of the partitions with an isolated point}
\label{subsec contribution of the partitions with an isolated point}

In this section, we still work with a fixed $p \geq 2$ and some fixed $\I \in \pa_p$. Lemma~\ref{lem upper bound} seems to say that the main contribution to the integral~\eqref{eq integral MI} should be given by terms of the form~\eqref{eq integral MIpJd} where $\norm{\J}$ is large, i.e.~$\J$ defines many small clusters. However, by looking more carefully at $\D_d^\I$, we can prove that the contribution of $M^{\I,p}_{\J,d}$ is also small if $\J$ contains an isolated point. In particular, for such a $\J$, the integral~\eqref{eq integral MIpJd} will only contribute an error term in the estimates of Theorem~\ref{thm main}. The main result of this section is Corollary~\ref{cor partitions with a singleton} below.

\begin{lem}
\label{lem partitions with a singleton}
Let $p \geq 2$ and let $\I \in \pa_p$. We assume that there exists $j \in \{1,\dots,p\}$ such that $\{j\} \in \I$. Then, for all $\J \in \pa_\I$ such that $\left\{\{j\}\right\} \in \J$, we have:
\begin{equation*}
\forall \underline{x}_\I \in M^{\I,p}_{\J,d}, \qquad \D^\I_d(\underline{x}_\I) = O(d^{\frac{p}{4}-1}),
\end{equation*}
as $d \to +\infty$, uniformly in $\underline{x}_\I \in M^{\I,p}_{\J,d}$.
\end{lem}

\begin{rem}
\label{rem partitions with a singletons}
This statement and its proof are typically easier to understand when $\I =\I_0(p)$. In this case, we can identify $\I$ with $\{1,\dots,p\}$ canonically as in Example~\ref{ex term I0}. Then $\J \in \pa_p$, and our hypothesis simply means that $\J$ contains a singleton $\{j\}$. The statement of the lemma becomes that: if $x=(x_i)_{1 \leq i \leq p} \in M^p$ is such that $x_j$ is far from the other components of $x$, meaning at distance greater than $b_p \frac{\ln d}{\sqrt{d}}$, then $\D_d^{\I_0(p)}(x) = O(d^{\frac{p}{4}-1})$.
\end{rem}

\begin{proof}[Proof of Lemma~\ref{lem partitions with a singleton}]
Recall that we defined $\mathcal{S}_\I$ in Definition~\ref{def subset adapted to I}. Since $\{j\} \in \I$, we can split $\mathcal{S}_\I$ into the subsets $A \in \mathcal{S}_\I$ that contain $j$, and those that do not. Moreover, $A \mapsto A \sqcup \{j\}$ is a bijection from $\{ A \in \mathcal{S}_\I \mid j \notin A\}$ to $\{ A \in \mathcal{S}_\I \mid j \in A\}$.

Let $\underline{x}_\I =(x_I)_{I \in \I} \in M^{\I,p}_{\J,d}$ and let $A \in \mathcal{S}_\I$ be such that $j \notin A$. We regroup the terms corresponding to $A$ and $A \sqcup\{j\}$ in the sum defining $\D_d^\I(\underline{x}_\I)$, see Lemma~\ref{lem integral expression}. Since we have $\I_{A \sqcup \{j\}} = \I_A \sqcup \{\{j\}\}$, we obtain:
\begin{multline}
\label{eq difference two terms}
(-1)^{p - \norm{A}} \mathcal{R}_d^{\norm{\I_A}}(\underline{x}_{\I_A}) \prod_{i \notin A} \mathcal{R}_d^1(x_{\{i\}}) + (-1)^{p - \norm{A}-1} \mathcal{R}_d^{\norm{\I_A}+1}(\underline{x}_{\I_A},x_{\{j\}}) \prod_{i \notin A \sqcup \{j\}} \mathcal{R}_d^1(x_{\{i\}})\\
= (-1)^{p - \norm{A}-1} \left(\mathcal{R}_d^{\norm{\I_A}+1}(\underline{x}_{\I_A},x_{\{j\}})-\mathcal{R}_d^{\norm{\I_A}}(\underline{x}_{\I_A})\mathcal{R}_d^1(x_{\{j\}})\right)\prod_{i \notin A \sqcup \{j\}} \mathcal{R}_d^1(x_{\{i\}}),
\end{multline}
where we used the fact that $\mathcal{R}_d^{\norm{\I_A}+1}$ is a symmetric function of its arguments, cf.~Definition~\ref{def Rkd}. Recall that $\mathcal{R}^0_d=1$ for all $d \geq 1$ by convention.

Since $\{\{j\}\} \in \J$, the point $x_{\{j\}}$ is far from the other coordinates of $\underline{x}_\I$ by definition of $M^{\I,p}_{\J,d}$, see Definition~\ref{def M AIdp}. More precisely, for all $I \in \I \setminus \{\{j\}\}$ we have $\rho_g(x_I,x_{\{j\}}) \geq b_p \frac{\ln d}{\sqrt{d}}$. Then, by Proposition~\ref{prop multiplicativity Rk}, the term on the right-hand side of Equation~\eqref{eq difference two terms} is $O(d^{\frac{1}{2}(\norm{\I_A}-\norm{A})+ \frac{p}{4}-1})$, uniformly in $\underline{x}_\I \in M^{\I,p}_{\J,d}$. Since $\I_A \in \pa_A$, we have $\norm{\I_A} \leq \norm{A}$. Hence, the previous term is $O(d^{\frac{p}{4}-1})$. We conclude the proof by summing this estimates over $\{ A \in \mathcal{S}_\I \mid j \notin A\}$.
\end{proof}

\begin{cor}
\label{cor partitions with a singleton}
Let $p \geq 2$ and $\I \in \pa_p$. We assume that there exists $j \in \{1,\dots,p\}$ such that $\{j\} \in \I$. Then, for all $\J \in \pa_\I$ such that $\left\{\{j\}\right\} \in \J$, we have:
\begin{equation*}
\forall \phi_1,\dots,\phi_p \in \mathcal{C}^0(M), \qquad \int_{\underline{x}_\I \in M^{\I,p}_{\J,d}} \iota_\I^*\phi(\underline{x}_\I)\, \D^\I_d(\underline{x}_\I) \rmes{M}^{\norm{\I}} = \left(\prod_{i=1}^p \Norm{\phi_i}_\infty\right) O(d^{\frac{p}{4}-1}),
\end{equation*}
where the error term does not depend on $(\phi_1,\dots,\phi_p)$.
\end{cor}

\begin{proof}
We obtain this corollary by integrating the estimate of Lemma~\ref{lem partitions with a singleton} over $M^{\I,p}_{\J,d}$, using the fact that $\Norm{\iota_\I^* \phi}_\infty \leq \prod_{i=1}^p \Norm{\phi_i}_\infty$.
\end{proof}


\subsection{Contribution of the partitions into pairs}
\label{subsec contribution of the partitions into pairs}

Let $p \geq 2$ and let $\phi_1,\dots,\phi_p$ be test-functions. Recall that $m_p(\nu_d)(\phi_1,\dots,\phi_p)$ is the sum of the integrals~\eqref{eq integral MIpJd} for $\I \in \pa_p$ and $\J \in \pa_\I$, and that our final goal is to prove Theorem~\ref{thm main}. This theorem gives the asymptotics of $m_p(\nu_d)(\phi_1,\dots,\phi_p)$ as $d$ goes to infinity, up to an error term of the form $O\!\left(d^{\frac{1}{2}\floor*{\frac{p-1}{2}}}(\ln d)^p\right)$. We proved in Lemma~\ref{lem upper bound} that the integral~\eqref{eq integral MIpJd} only contributes an error term if $\norm{\J} < \frac{p}{2}$. Besides, we proved in Corollary~\ref{cor partitions with a singleton} that the integral~\eqref{eq integral MIpJd} also contributes an error term if there exists $j$ such that $\{\{j\}\} \in \J$, in particular if $\norm{\J}$ is too large. This last point will be made more precise in Section~\ref{subsec proof of the main theorem} below.

In this section, we study the integrals of the form~\eqref{eq integral MIpJd} that will contribute to the leading term in the asymptotics of $m_p(\nu_d)(\phi_1,\dots,\phi_p)$. These integrals are indexed by couples of partitions $(\I,\J)$ satisfying the following technical condition: the double partitions into pairs. Recall that we denoted by $\pp_A$ the set of partition into pairs of the finite set $A$, see Definition~\ref{def partitions}, and that we denoted by $\I_S$ the partition of $S \subset A$ induced by a partition $\I \in \pa_A$, see~Definition~\ref{def induced partition}.

\begin{dfn}[Double partition into pairs]
\label{def CA}
Let $A$ be a finite set, let $\I \in \pa_A$ and let $\J \in \pa_\I$. We say that a couple $(\I,\J)$ is a \emph{double partition into pairs} of $A$ if there exists $S \subset A$ such that:
\begin{enumerate}
\item \label{cond 1} $\I_S = \{\{s\} \mid s \in S\}$ and $\I \setminus \I_S = \I_{A \setminus S} \in \pp_{A \setminus S}$;
\item \label{cond 2} $\J_{\I_S} \in \pp_{\I_S}$ and $\J \setminus \J_{\I_S} = \J_{\I \setminus \I_S} = \{\{I\} \mid I \in \I \setminus \I_S\}$.
\end{enumerate}
We denote by $\mathfrak{C}_A$ the set of such double partitions into pairs of $A$. If $A = \{1,\dots,p\}$ we simply denote $\mathfrak{C}_p = \mathfrak{C}_A$.
\end{dfn}

Let us take some time to comment upon this definition. Note that we will mostly use it with $A = \{1,\dots,p\}$. However, the general case will be useful in Lemmas~\ref{lem partitions into pairs} and~\ref{lem partitions into pairs 2} below. Condition~\ref{cond 1} in Definition~\ref{def CA} simply means that $\I$ only contains singletons and pairs. The singletons are exactly those of the form $\{s\}$ with $s \in S$, and the pairs are formed of elements of $A \setminus S$. Note that $S$ is the union of the singletons in~$\I$. In particular, it is uniquely defined by $(\I,\J)$. Similarly, Condition~\ref{cond 2} means that $\J$ only contains singletons and pairs with some additional constraints. The singletons in $\J$ are exactly of the form $\{I\}$ where $I \notin \I_S$ (i.e.~$I$ is itself a pair of elements of $A \setminus S$) and the pairs are of the form $\{\{s\},\{t\}\}$ with $s,t \in S$ distinct. That is, $\J$ contains only pairs of singletons and singletons of pairs\dots

\begin{rem}
\label{rem conditions 1 and 2}
Let $A$ be a finite set and let $(\I,\J) \in \mathfrak{C}_A$. Let $S$ be as in Definition~\ref{def CA}. We have $\norm{S} = \norm{\I_S}$, and since $\I_S$ admits a partition $\J_{\I_S}$ into pairs, this cardinality is even. Similarly, since $\I \setminus \I_S$ is a partition into pairs of $A \setminus S$, we have $\norm{A} = \norm{S} + 2\norm{\I \setminus \I_S}$. Hence, $\norm{A}$ is even. In particular, $\mathfrak{C}_A$ is empty if $\norm{A}$ is odd.
\end{rem}

\begin{exa}
\label{ex CA}
Let $A$ be a finite set of even cardinality, the following are examples of double partitions into pairs showing that $\mathfrak{C}_A$ is not empty.
\begin{itemize}
\item Taking $S = \emptyset$, we have $\I_S = \emptyset$ and $\I = \I \setminus \I_S$ can be chosen as any element of $\pp_A$. Then there is only one possibility for $\J$, we must have $\J = \J_{\I \setminus \I_S} = \{\{I\} \mid I \in \I\}$.
\item Taking $S = A$, we must have $\I = \I_S = \{\{a\} \mid a \in A\}$. Identifying canonically $\I$ with $A$, we can choose any $\J \in \pp_\I \simeq \pp_A$.
\end{itemize}
\end{exa}

In the remainder of this section, we study the large $d$ asymptotics of integrals of the form~\eqref{eq integral MIpJd} where $(\I,\J) \in \mathfrak{C}_p$. For this, we need to show that the function $\D_d^\I$ defined in Lemma~\ref{lem integral expression} factorizes nicely on $M^{\I,p}_{\J,d}$, up to an error term. The first step is Lemma~\ref{lem partitions into pairs} where we factor the contribution from the pairs of singletons in $\J$, i.e~the elements of $\J_{\I_S}$.

\begin{lem}
\label{lem partitions into pairs}
Let $A$ be a finite set. For all $(\I,\J) \in \mathfrak{C}_A$, let $S \subset A$ be as in Definition~\ref{def CA}. Then, the following holds uniformly for all $\underline{x}_\I \in M^{\I,p}_{\J,d}$:
\begin{equation*}
\D_d^\I(\underline{x}_\I) = \D_d^{\I \setminus \I_S}(\underline{x}_{\I \setminus \I_S}) \left(\prod_{J \in \J_{\I_S}}\D_d^J(\underline{x}_J)\right) +O(d^{\frac{\norm{\I}}{2}-\frac{p}{4}-1}).
\end{equation*}
\end{lem}

\begin{proof}
Recall that, as discussed in Remark~\ref{rem conditions 1 and 2}, the set $\mathfrak{C}_A$ is empty if $\norm{A}$ is odd. Hence Lemma~\ref{lem partitions into pairs} is true in this case. If $\norm{A}$ is even, the proof is by induction on (one half of) $\norm{A}$.

\subparagraph*{Base case.}
If $\norm{A}=0$, we have $A = \emptyset$, hence $\I = \emptyset$, hence $\J = \emptyset$. Recalling the definition of $\D^\I_d$ in Lemma~\ref{lem integral expression} and our convention that $\mathcal{R}_d^0=1$, we have $\D^\I_d =1$ and similarly $\D^{\I\setminus \I_S}_d =1$. The product over $\J_{\I_S}$ is indexed by the empty set, hence also equal to $1$. Thus, the result is tautologically true.

\subparagraph*{Inductive step.}
Let $A$ be a finite set of even positive cardinality. Let us assume that the result of Lemma~\ref{lem partitions into pairs} holds for any finite set $B$ of cardinality $\norm{A}-2$. Let $(\I,\J) \in \mathfrak{C}_A$ and let $S$ be as in Definition~\ref{def CA}. Let $J \in \J_{\I_S}$, there exists $s,t \in S$ distinct such that $J =\{\{s\},\{t\}\}$. The key part of the proof is to show that, the following holds uniformly for all $\underline{x}_\I \in M^{\I,p}_{\J,d}$:
\begin{equation}
\label{eq inductive step}
\D^\I_d(\underline{x}_\I) = \D_d^{\I \setminus J}(\underline{x}_{\I \setminus J})\D_d^J(\underline{x}_J) +O(d^{\frac{\norm{\I}}{2}-\frac{p}{4}-1}).
\end{equation}

Let us assume for now that Equation~\eqref{eq inductive step} holds. Let us denote by $B = A \setminus \{s,t\}$. Then we have $\I \setminus J = \I \setminus \{\{s\},\{t\}\} \in \pa_B$ and $\J \setminus \{J\} \in \pa_{\I \setminus J}$. One can check that $(\I \setminus J, \J \setminus \{J\}) \in \mathfrak{C}_B$. Indeed, the partition $\I \setminus J$ only contains pairs and singletons, and the union of its singletons is $T = S \setminus \{s,t\}$. Then we have $(\I \setminus J)_T = \I_S \setminus J$, and $(\J \setminus \{J\})_{\I_S \setminus J} = \J_{\I_S} \setminus \{J\}$ is a partition into pairs of $\I_S \setminus J$. Moreover, the pairs in $\I \setminus J$ are the same as the pairs in $\I$. Since $\norm{B} = \norm{A}-2$, we can apply the induction hypothesis for $(\I \setminus J, \J \setminus \{J\}) \in \mathfrak{C}_B$.

Note that $\I \setminus J$ is the union of the elements of $\J \setminus \{J\}$. Then, by Lemma~\ref{lem M AIdp restricted}, if $\underline{x}_\I \in M^{\I,p}_{\J,d}$ we have $\underline{x}_{\I \setminus J} \in M^{\I \setminus J,p}_{\J \setminus \{J\},d}$. Hence, by induction, the following holds uniformly for all $\underline{x}_\I \in M^{\I,p}_{\J,d}$:
\begin{equation*}
\D_d^{\I \setminus J}(\underline{x}_{\I \setminus J}) = \D_d^{\I \setminus \I_S}(\underline{x}_{\I \setminus \I_S}) \left(\prod_{K \in \J_{\I_S} \setminus \{J\}}\D_d^K(\underline{x}_K)\right) +O(d^{\frac{\norm{\I}}{2}-\frac{p}{4}-2}).
\end{equation*}
We use this relation in Equation~\eqref{eq inductive step}. Since $\D_d^J$ is uniformly $O(d)$ by Proposition~\ref{prop bounded density}, the conclusion of Lemma~\ref{lem partitions into pairs} is satisfied for $(\I,\J) \in \mathfrak{C}_A$. This concludes the inductive step.

\subparagraph*{Proof of Equation~\eqref{eq inductive step}.}
In order to complete the proof of Lemma~\ref{lem partitions into pairs}, we need to prove that Equation~\eqref{eq inductive step} holds. The proof of this fact is in the same spirit as the proof of Lemma~\ref{lem partitions with a singleton}. We start from the definition of $\D_d^\I(\underline{x}_\I)$, given in  Lemma~\ref{lem integral expression}, and regroup the terms in the sum in fours. For each of these quadruples of summands we apply Proposition~\ref{prop multiplicativity Rk}, using the fact that $\underline{x}_\I \in M^{\I,p}_{\J,d}$. This yields the result, keeping in mind the uniform upper bound of Proposition~\ref{prop bounded density}.

Recall that $\D_d^\I$ is defined as a sum indexed by $\mathcal{S}_\I$, and that the elements of $\mathcal{S}_\I$ are the subsets of $\{1,\dots,p\}$ adapted to $\I$ (see Definition~\ref{def subset adapted to I}). Since $\{s\}$ and $\{t\} \in \I$, we can decompose $\mathcal{S}_\I$ as the following disjoint union:
\begin{equation*}
\mathcal{S}_\I = \bigsqcup_{\{A \in \mathcal{S}_\I \mid s \notin A, t \notin A\} } \left\{A, A \sqcup \{s\}, A \sqcup \{t\}, A \sqcup \{s,t\}\right\}.
\end{equation*}

Let $A \in \mathcal{S}_\I$ be such that $A \subset \{1,\dots,p\} \setminus \{s,t\}$. Given $\underline{x}_\I \in M^\I$, we regroup the four terms corresponding to $A$, $A \sqcup \{s\}$, $A \sqcup \{t\}$ and $A \sqcup \{s,t\}$ in the sum defining $\D_d^\I(\underline{x}_\I)$. Keeping in mind that the Kac--Rice densities $(\mathcal{R}_d^k)_{k \geq 1}$ are symmetric functions of their arguments, we obtain:
\begin{equation}
\label{eq difference four terms}
\begin{aligned}
(-1)^{p-\norm{A}} &\left(\mathcal{R}_d^{\norm{\I_A}+2}(\underline{x}_{\I_A},x_{\{s\}},x_{\{t\}}) - \mathcal{R}_d^{\norm{\I_A}+1}(\underline{x}_{\I_A},x_{\{s\}})\mathcal{R}^1_d(x_{\{t\}})\right.\\
&\left.- \mathcal{R}_d^{\norm{\I_A}+1}(\underline{x}_{\I_A},x_{\{t\}})\mathcal{R}^1_d(x_{\{s\}}) + \mathcal{R}_d^{\norm{\I_A}}(\underline{x}_{\I_A})\mathcal{R}^1_d(x_{\{s\}})\mathcal{R}^1_d(x_{\{t\}})\right) \prod_{i \notin A \sqcup \{s,t\}} \mathcal{R}^1_d(x_{\{i\}})
\end{aligned}
\end{equation}
If $\underline{x}_\I \in M^{\I,p}_{\J,d}$, then for any $I \in \I \setminus \{\{s\},\{t\}\}$ we have $\rho_g(x_I,x_{\{s\}}) \geq b_p\frac{\ln d}{\sqrt{d}}$ and $\rho_g(x_I,x_{\{t\}}) \geq b_p\frac{\ln d}{\sqrt{d}}$ by definition. Applying Propositions~\ref{prop bounded density} and~\ref{prop multiplicativity Rk} in Equation~\eqref{eq difference four terms}, we obtain:
\begin{equation*}
(-1)^{p-\norm{A}} \left(\prod_{i \notin A \sqcup \{s,t\}} \mathcal{R}^1_d(x_{\{i\}})\right)\mathcal{R}_d^{\norm{\I_A}}(\underline{x}_{\I_A}) \D_d^J(\underline{x}_J)+O\!\left(d^{\frac{\norm{\I_A}+p-\norm{A}}{2}-\frac{p}{4}-1}\right).
\end{equation*}
Since $A$ is adapted to $\I$, we have $\I = \I_A \sqcup \{\{i\} \mid i \notin A\}$ by Definition~\ref{def subset adapted to I}. This implies that $\norm{\I} =\norm{\I_A} +p - \norm{A}$. Hence the error term in the previous equation is $O(d^{\frac{\norm{\I}}{2}-\frac{p}{4}-1})$. Summing these terms over the subsets $A \in \mathcal{S}_\I$ such that $A \subset \{1,\dots,p\} \setminus \{s,t\}$, we proved that Equation~\eqref{eq inductive step} is satisfied uniformly for all $\underline{x}_\I \in M^{\I,p}_{\J,d}$.
\end{proof}

With the same notations as in Lemma~\ref{lem partitions into pairs}, the second step in our factorization of $\D^\I_d$ is to factor the contributions from the singletons in $\J$, that is from the $\{\{I\} \mid  I \in \I \setminus \I_S\}$. This is the purpose of the following lemma.

\begin{lem}
\label{lem partitions into pairs 2}
Let $A$ be a finite set, let $(\I,\J) \in \mathfrak{C}_A$ and let $S \subset A$ be as in Definition~\ref{def CA}. Then, the following holds uniformly for all $\underline{x}_\I \in M^{\I,p}_{\J,d}$:
\begin{equation*}
\D_d^\I(\underline{x}_\I) = \left(\prod_{I \in \I \setminus \I_S} \mathcal{R}_d^1(x_I)\right) \left(\prod_{J \in \J_{\I_S}}\D_d^J(\underline{x}_J)\right) +O(d^{\frac{\norm{\I}}{2}-\frac{p}{4}-1}).
\end{equation*}
\end{lem}

\begin{proof}
Applying first Lemma~\ref{lem partitions into pairs}, uniformly for all $\underline{x}_\I \in M^{\I,p}_{\J,d}$ we have:
\begin{equation}
\label{eq part into pairs 1}
\D_d^\I(\underline{x}_\I) = \D_d^{\I \setminus \I_S}(\underline{x}_{\I \setminus \I_S}) \left(\prod_{J \in \J_{\I_S}}\D_d^J(\underline{x}_J)\right) +O(d^{\frac{\norm{\I}}{2}-\frac{p}{4}-1}).
\end{equation}
By Condition~\ref{cond 1} in Definition~\ref{def CA}, the partition $\I \setminus \I_S$ is a partition into pairs of $A \setminus S$. Recalling Definition~\ref{def subset adapted to I}, we have $\mathcal{S}_{\I \setminus \I_S} = \{A \setminus S\}$ hence $\D_d^{\I \setminus \I_S} = \mathcal{R}_d^{\norm{\I \setminus \I_S}}$.

Let $\underline{x}_\I \in M^{\I,p}_{\J,d}$, by Condition~\ref{cond 2} in Definition~\ref{def CA}, for all $I \in \I \setminus \I_S$ we have $\{I\} \in \J$, hence $x_I$ is far from the other components of $\underline{x}_\I$. In particular, for all $I,J \in \I \setminus \I_S$ distinct we have $\rho_g(x_I,x_J) \geq b_p \frac{\ln d}{\sqrt{d}}$. Applying Proposition~\ref{prop multiplicativity Rk} several times, we obtain:
\begin{equation}
\label{eq part into pairs 2}
\D_d^{\I \setminus \I_S}(\underline{x}_{\I \setminus \I_S}) = \mathcal{R}_d^{\norm{\I \setminus \I_S}}(\underline{x}_{\I \setminus \I_S}) = \prod_{I \in \I \setminus \I_S} \mathcal{R}_d^1(x_I) + O(d^{\frac{\norm{\I \setminus \I_S}}{2}-\frac{p}{4}-1}),
\end{equation}
where the error term is obtained by using the uniform upper bound of Proposition~\ref{prop bounded density}.

By Proposition~\ref{prop bounded density} once again, for any $J \in \J_{\I_S}$ we have $\D_d^J(\underline{x}_J) = O(d)$ uniformly on $M^J$. Since, $\J_{\I_S}$ is a partition into pairs of $\I_S$, we have $\norm{\I_S} = 2\norm{\J_{\I_S}}$. Hence, we have uniformly:
\begin{equation}
\label{eq part into pairs 3}
\prod_{J \in \J_{\I_S}}\D_d^J(\underline{x}_J) = O(d^\frac{\norm{\I_S}}{2}).
\end{equation}
Since $\norm{\I} = \norm{\I_S} + \norm{\I \setminus \I_S}$, the result follows from Equations~\eqref{eq part into pairs 1}, \eqref{eq part into pairs 2} and~\eqref{eq part into pairs 3}.
\end{proof}

We can finally derive the asymptotics of the integrals of the form~\eqref{eq integral MIpJd} for $(\I,\J) \in \mathfrak{C}_p$, which is the main result of this section.

\begin{lem}
\label{lem leading term}
Let $p \geq 2$, let $(\I,\J) \in \mathfrak{C}_p$ and let $S$ be as in Definition~\ref{def CA}. Then, for any $\phi_1,\dots,\phi_p \in \mathcal{C}^0(M)$, we have:
\begin{multline*}
\int_{\underline{x}_\I \in M^{\I,p}_{\J,d}}\!\iota_\I^*\phi(\underline{x}_\I) \D^\I_d(\underline{x}_\I) \rmes{M}^{\norm{\I}}= \!\left(\prod_{\{i,j\} \in \I \setminus \I_S} \int_M \phi_i\phi_j\mathcal{R}^1_d \rmes{M}\right)\!\left(\prod_{J \in \J_{\I_S}} \int_{M^J} \!\phi_J \D_d^J \rmes{M}^2 \right)\\+ \left(\prod_{i=1}^p \Norm{\phi_i}_\infty\right) O\!\left(d^{\frac{1}{2}\floor*{\frac{p-1}{2}}}(\ln d)^p\right),
\end{multline*}
where the constant involved in the $O\!\left(d^{\frac{1}{2}\floor*{\frac{p-1}{2}}}(\ln d)^p\right)$ is independent of $(\phi_1,\dots,\phi_p)$.
\end{lem}

\begin{proof}
The idea of the proof is to use the factorization obtained in Lemma~\ref{lem partitions into pairs 2} for the density function $\D^\I_d$. The test-function $\iota_\I^*\phi$ also factorizes, hence we can write the integrand $(\iota_\I^*\phi) \D^\I_d$ as a nice product. The difficulty is that the domain of integration $M^{\I,p}_{\J,d}$ is not a Cartesian product. To deal with this, we use Lemma~\ref{lem M AIdp blurred} to show that we can replace $M^{\I,p}_{\J,d}$ by a product of domains of dimension $1$ or $2$, up to an error term.

\subparagraph*{Factorization of the integrand.}
Recall that, by Condition~\ref{cond 1} in Definition~\ref{def CA}, the partition $\I$ only contains singletons and pairs, and $S$ is the union of the singletons of $\I$. Let $\underline{x}_\I \in M^\I$, we have:
\begin{equation*}
\iota_\I^*\phi(\underline{x}_\I)\ = \phi(\iota_\I(\underline{x}_\I)) = \prod_{I \in \I} \prod_{i \in I} \phi_i(x_I) = \left(\prod_{\{i,j\}=I \in \I \setminus \I_S} \phi_i(x_I)\phi_j(x_I)\right) \left(\prod_{\{i\} \in \I_S} \phi_i(x_{\{i\}})\right).
\end{equation*}
Let $F_d^\I: M^\I \to \R$ be the function defined by:
\begin{equation}
\label{eq def FId}
F_d^\I(\underline{x}_\I) = \left(\prod_{\{i,j\}=I \in \I \setminus \I_S} \phi_i(x_I)\phi_j(x_I)\mathcal{R}^1_d(x_I)\right)\left(\prod_{J \in \J_{\I_S}} \phi_J(\underline{x}_J)\D_d^J(\underline{x}_J)\right)
\end{equation}
for all $\underline{x}_\I \in M^\I$. Here we use Notation~\ref{ntn product indexed by A}, hence for any $J = \{\{i\},\{j\}\} \in \J$ we have:
\begin{equation*}
\phi_J(\underline{x}_J) = \phi_{\{i\}}(x_{\{i\}})\phi_{\{j\}}(x_{\{j\}}) = \phi_i(x_{\{i\}})\phi_j(x_{\{j\}}).
\end{equation*}

Note that since $\I \in \pa_p$, we have $\norm{\I} \leq p$ so that $\frac{\norm{\I}}{2}-\frac{p}{4}-1 \leq \frac{p}{4}-1 < \frac{1}{2}\floor*{\frac{p-1}{2}}$. Applying Lemma~\ref{lem partitions into pairs 2}, we obtain:
\begin{equation}
\label{eq leading term}
\iota_\I^*\phi(\underline{x}_\I) \D^\I_d(\underline{x}_\I) =  F_d^\I(\underline{x}_I)+\left(\prod_{i=1}^p \Norm{\phi_i}_\infty\right) o(d^{\frac{1}{2}\floor*{\frac{p-1}{2}}}),
\end{equation}
where the error term does not depend on $\underline{x}_\I$ or $(\phi_i)_{1 \leq i \leq p}$. Thus, up to an error term, the quantity we are interested in is the integral of $F_d^\I$ over $M^{\I,p}_{\J,d}$.

\subparagraph*{Changing the domain of integration.}
By Lemma~\ref{lem M AIdp blurred}, we have:
\begin{equation*}
M^{\I,p}_{\J,d} \subset \prod_{J \in \J} M^{J,p}_{\{J\},d} \subset \bigsqcup_{\K \geq \J} M^{\I,p}_{\K,d},
\end{equation*}
hence
\begin{equation}
\label{eq domain of integration}
\prod_{J \in \J} M^{J,p}_{\{J\},d} = M^{\I,p}_{\J,d} \sqcup \bigsqcup_{\K > \J} \left(M^{\I,p}_{\K,d} \cap \prod_{J \in \J} M^{J,p}_{\{J\},d}\right).
\end{equation}

We want to prove that the integral of $F_d^\I$ over $M^{\I,p}_{\J,d}$ is equal to its integral over $\prod_{J \in \J} M^{J,p}_{\{J\},d}$, up to an error term. Let $\K \in \pa_\I$ be such that $\K > \J$. We can bound the integral of $F_d^\I$ over $\Omega_\K = M^{\I,p}_{\K,d} \cap \prod_{J \in \J} M^{J,p}_{\{J\},d}$, using the same method as in the proof of Lemma~\ref{lem upper bound}. We have:
\begin{equation*}
\norm{\int_{\underline{x}_\I \in \Omega_\K} F_d^\I(\underline{x}_\I) \rmes{M}^{\norm{\I}}} \leq \int_{\underline{x}_\I \in \Omega_\K} \norm{F_d^\I(\underline{x}_\I)} \rmes{M}^{\norm{\I}} \leq \int_{\underline{x}_\I \in M^{\I,p}_{\K,d}} \norm{F_d^\I(\underline{x}_\I)} \rmes{M}^{\norm{\I}}.
\end{equation*}
Using Proposition~\ref{prop bounded density} in Equation~\eqref{eq def FId}, we prove that the function $F_d^\I$ is bounded uniformly over~$M^\I$ by $\left(\prod_{i=1}^p \Norm{\phi_i}_\infty\right) O(d^\frac{\norm{\I}}{2})$. Then, by Lemma~\ref{lem volume M AIdp}, we have:
\begin{equation}
\label{eq error term K}
\norm{\int_{\underline{x}_\I \in \Omega_\K} F(\underline{x}_\I) \rmes{M}^{\norm{\I}}} \leq \Norm{F^\I_d}_\infty \vol{M^{\I,p}_{\K,d}} = \left(\prod_{i=1}^p \Norm{\phi_i}_\infty\right) O\!\left(d^{\frac{\norm{\K}}{2}}(\ln d)^{\norm{\I}-\norm{\K}}\right).
\end{equation}

Recall that $\K > \J$, so that $\norm{\K} < \norm{\J}$, as discussed after Definition~\ref{def order on partitions}. Since $(\I,\J) \in \mathfrak{C}_p$, by~Definition~\ref{def CA} we have:
\begin{equation*}
\norm{\J} = \frac{1}{2}\norm{\I_S} + \norm{\I \setminus \I_S} = \frac{1}{2}\norm{S} + \frac{1}{2}(p - \norm{S}) = \frac{p}{2}.
\end{equation*}
Thus, we have $\norm{\K} \leq \norm{\J}-1 = \frac{p}{2}-1 \leq \floor*{\frac{p-1}{2}}$. On the other hand, $\norm{\I} - \norm{\K} \leq \norm{\I} \leq p$. Finally, using Equations~\eqref{eq domain of integration} and~\eqref{eq error term K}, we obtain:
\begin{equation}
\label{eq integral FId}
\int_{M^{\I,p}_{\J,d}} F_d^\I \rmes{M}^{\norm{\I}} = \int_{\prod_{J \in \J} M^{J,p}_{\{J\},d}} F_d^\I \rmes{M}^{\norm{\I}} +\left(\prod_{i=1}^p \Norm{\phi_i}_\infty\right)O\!\left(d^{\frac{1}{2}\floor*{\frac{p-1}{2}}}(\ln d)^p\right).
\end{equation}

By Condition~\ref{cond 2} in Definition~\ref{def CA}, if $J \in \J \setminus \J_{\I_S}$ then $J = \{I\}$ for some $I \in \I \setminus \I_S$, so that $M^{J,p}_{\{J\},d} = M$. Bearing this fact in mind, Equations~\eqref{eq def FId}, \eqref{eq leading term} and~\eqref{eq integral FId} yield:
\begin{equation}
\label{eq integral over J}
\begin{aligned}
\int_{\underline{x}_\I \in M^{\I,p}_{\J,d}} \!\iota_\I^*\phi(\underline{x}_\I) \D^\I_d(\underline{x}_\I) \rmes{M}^{\norm{\I}}\!= \left(\prod_{\{i,j\} \in \I \setminus \I_S} \int_M \phi_i\phi_j\mathcal{R}^1_d \rmes{M}\right)&\!\left(\prod_{J \in \J_{\I_S}} \int_{M^{J,p}_{\{J\},d}}\!\phi_J \D_d^J \rmes{M}^2 \right)\\
+ \left(\prod_{i=1}^p \Norm{\phi_i}_\infty\right)& O\!\left(d^{\frac{1}{2}\floor*{\frac{p-1}{2}}}(\ln d)^p\right).
\end{aligned}
\end{equation}

\subparagraph*{Conclusion of the proof.}
In order to conclude the proof, we need to replace the integral over $M^{J,p}_{\{J\},d}$ by an integral over $M^J$ in Equation~\eqref{eq integral over J}, for all $J \in \J_{\I_S}$.

Let $J \in \J_{\I_S}$, there exists $i$ and $j \in S$ such that $J = \{\{i\},\{j\}\}$. Then, $\pa_J$ has exactly two elements: $\{J\}$ and $\I_0(J) = \left\{\rule{0em}{2.2ex}\{\{i\}\},\{\{j\}\}\right\}$. By Proposition~\ref{prop multiplicativity Rk} and the definition of $\D_d^J$ (see Lemma~\ref{lem integral expression}), we have the following uniform estimate for all $\underline{x}_J \in M^{J,p}_{\I_0(J),d}$:
\begin{equation*}
\D_d^J(\underline{x}_J) = \mathcal{R}^2_d(x_{\{i\}},x_{\{j\}}) - \mathcal{R}^1_d(x_{\{i\}})\mathcal{R}^1_d(x_{\{j\}}) = O(d^{-\frac{p}{4}}).
\end{equation*}
Since $M^J = M^{J,p}_{\{J\},d} \sqcup M^{J,p}_{\I_0(J),d}$, we obtain:
\begin{equation*}
\int_{M^{J,p}_{\{J\},d}} \phi_J \D_d^J \rmes{M}^2 = \int_{M^J} \phi_J \D_d^J \rmes{M}^2 + \Norm{\phi_J}_\infty O(d^{-\frac{p}{4}}).
\end{equation*}
Then, using one last time the upper bound of Proposition~\ref{prop bounded density}, the leading term on the right-hand side of Equation~\eqref{eq integral over J} equals:
\begin{equation*}
\left(\prod_{\{i,j\} \in \I \setminus \I_S} \int_M \phi_i\phi_j\mathcal{R}^1_d \rmes{M}\right)\!\left(\prod_{J \in \J_{\I_S}} \int_{M^J} \phi_J \D_d^J \rmes{M}^2 \right) + \left(\prod_{i=1}^p \Norm{\phi_i}_\infty\right) O(d^{\frac{\norm{\I \setminus \I_S}}{2}+\norm{\J_{\I_S}}-1-\frac{p}{4}}).
\end{equation*}
The conclusion follows using that $(\I,\J) \in\mathfrak{C}_p$. Indeed, by Definition~\ref{def CA}, we have:
\begin{equation*}
\frac{\norm{\I \setminus \I_S}}{2}+\norm{\J_{\I_S}}-1-\frac{p}{4} = \frac{p-\norm{S}}{4} + \frac{\norm{\I_S}}{2} - \frac{p}{4} -1 = \frac{\norm{S}}{4}-1 \leq \frac{p}{4}-1  < \frac{1}{2}\floor*{\frac{p-1}{2}}.\qedhere
\end{equation*}
\end{proof}


\subsection{Proof of Theorem~\ref{thm main}}
\label{subsec proof of the main theorem}

In this section, we conclude the proof of our main result, that is the moments estimates of Theorem~\ref{thm main}. The key argument is that the leading term in the asymptotics of the $m_p(\nu_d)(\phi_1,\dots,\phi_p)$ comes from integrals of the form~\eqref{eq integral MIpJd}, where $(\I,\J) \in \mathfrak{C}_p$ is one of the double partitions into pairs studied in Section~\ref{subsec contribution of the partitions into pairs}. The integrals of the form~\eqref{eq integral MIpJd} with $(\I,\J) \notin \mathfrak{C}_p$ only contribute an error term. This last fact is proved in the following lemma.

\begin{lem}
\label{lem order of the error terms}
Let $p \geq 2$, let $\I \in \pa_p$ and let $\J \in \pa_\I$ be such that $(\I,\J) \notin \mathfrak{C}_p$. Then, for all $\phi_1,\dots,\phi_p \in \mathcal{C}^0(M)$ we have:
\begin{equation*}
\int_{\underline{x}_\I \in M^{\I,p}_{\J,d}} \iota_\I^*\phi(\underline{x}_\I) \D^\I_d(\underline{x}_\I) \rmes{M}^{\norm{\I}} = \left(\prod_{i=1}^p \Norm{\phi_i}_\infty \right)O\!\left(d^{\frac{1}{2}\floor*{\frac{p-1}{2}}}(\ln d)^p\right),
\end{equation*}
where the error term does not depend on $(\phi_i)_{1 \leq i \leq p}$.
\end{lem}

\begin{proof}
Since $\frac{p}{4}-1 = \frac{1}{2}\left(\frac{p-2}{2} -1\right) < \frac{1}{2} \floor*{\frac{p-1}{2}}$, we already know that the result holds if there exists $i \in \{1,\dots,p\}$ such that $\{i\} \in \I$ and $\{\{i\}\} \in \J$, by Corollary~\ref{cor partitions with a singleton}. Assuming that this is not the case, we will show in Lemma~\ref{lem combinatorics IJ} below that: if $\I \in \pa_p$ and $\J \in \pa_\I$, then $\norm{\J} \leq \frac{p}{2}$, and equality holds if and only if $(\I,\J) \in \mathfrak{C}_p$. Since we assumed that $(\I,\J) \notin \mathfrak{C}_p$, we have $\norm{\J} < \frac{p}{2}$. Moreover, since~$\norm{\J}$ is an integer, we have in fact $\norm{\J} \leq \floor*{\frac{p-1}{2}}$. Then, the result follows from Lemma~\ref{lem upper bound}, since $\norm{\I} \leq p$.
\end{proof}

\begin{lem}
\label{lem combinatorics IJ}
Let $p \geq 2$, let $\I \in \pa_p$ and let $\J \in \pa_\I$. We assume that, for all $i \in \{1,\dots,p\}$, if $\{i\} \in \I$ then $\{\{i\}\} \notin \J$. Then $\norm{\J} \leq \frac{p}{2}$. Moreover, $\norm{\J} = \frac{p}{2}$ if and only if $(\I,\J) \in \mathfrak{C}_p$.
\end{lem}

\begin{proof}
We will use notations consistent with those in Definition~\ref{def CA}. Let $\I_S=\{I \in \I, \norm{I} =1 \}$ denote the set of singletons in $\I$. We denote by $S= \bigsqcup_{I \in \I_S} I \subset \{1,\dots,p\}$, so that $\I_S$ is indeed the partition of $S$ induced by $\I$. Then $\I \setminus \I_S$ is the partition of $\{1,\dots,p\} \setminus S$ induced by $\I$, and its elements have cardinality at least $2$. Since $\{1,\dots,p\} = \bigsqcup_{I \in \I_S} I \sqcup \bigsqcup_{I \in \I \setminus \I_S} I$, we have:
\begin{equation}
\label{eq inequality p}
\norm{\I_S} + 2\norm{\I \setminus \I_S} \leq p.
\end{equation}
Moreover, equality holds in Equation~\eqref{eq inequality p} if and only if $\I \setminus \I_S$ is a partition into pairs, that is if and only if Condition~\ref{cond 1} of Definition~\ref{def CA} is satisfied.

Let us denote by $\I' = \{I \in \I \mid \{I\} \in \J\}$. Under our assumptions, if $I \in \I'$ it can not be a singleton, i.e.~$\I' \subset \I \setminus \I_S$. Hence, we have:
\begin{equation}
\label{eq inequality I'}
\norm{\I'} \leq \norm{\I \setminus \I_S},
\end{equation}
and equality holds in Equation~\eqref{eq inequality I'} if and only if $\I' = \I \setminus \I_S$.

By definition of $\I'$, we have $\J_{\I'} = \{\{I\} \mid I \in \I'\} \subset \J$. Moreover, $\J_{\I \setminus \I'} = \J \setminus \J_{\I'}$ is partition of $\I \setminus \I'$ in subsets of cardinality at least $2$. This yields:
\begin{equation}
\label{eq inequality J'}
\norm{\J} = \norm{\J_{\I'}} + \norm{\J \setminus \J_{\I'}} = \norm{\I'} + \norm{\J_{\I \setminus \I'} }\leq \norm{\I'} + \frac{1}{2}\left(\norm{\I}-\norm{\I'}\right) = \frac{1}{2}\left(\norm{\I}+\norm{\I'}\right).
\end{equation}
Moreover, equality holds in Equation~\eqref{eq inequality J'} if and only if $\J_{\I \setminus \I'}$ is a partition into pairs of $\I \setminus \I'$.

By Equations~\eqref{eq inequality p}, \eqref{eq inequality I'} and~\eqref{eq inequality J'}, we have:
\begin{equation}
\label{eq inequality J}
\norm{\J} \leq \frac{1}{2}\left(\norm{\I}+\norm{\I'}\right) \leq \frac{1}{2}\left(\norm{\I}+\norm{\I \setminus \I_S}\right) = \frac{1}{2}\left(\norm{\I_S} + 2 \norm{\I \setminus \I_S}\right)  \leq \frac{p}{2}.
\end{equation}
Moreover, we have $\norm{\J} = \frac{p}{2}$ if and only if equality holds in~\eqref{eq inequality p}, \eqref{eq inequality I'} and~\eqref{eq inequality J'}. Equality in~\eqref{eq inequality p} means that Condition~\ref{cond 1} in Definition~\ref{def CA} is satisfied. If equality holds in~\eqref{eq inequality I'} then $\I' = \I \setminus \I_S$. If in addition equality holds in Equation~\eqref{eq inequality J'}, we have that $\J_{\I \setminus \I'} = \J_{\I_S}$ is a partition of $\I_S$ into pairs and that $\J_{\I \setminus \I_S} = \J_{\I'} = \{\{I\} \mid I \in \I \setminus \I_S\} \subset \J$. Hence, Condition~\ref{cond 2} in Definition~\ref{def CA} is satisfied. Finally, if $\norm{\J} = \frac{p}{2}$ then $(\I,\J) \in \mathfrak{C}_p$. Conversely, if $(\I,\J) \in \mathfrak{C}_p$, equality holds in Equations~\eqref{eq inequality p}, \eqref{eq inequality I'} and~\eqref{eq inequality J'} and $\norm{\J} = \frac{p}{2}$.
\end{proof}

We conclude this section with the proof of Theorem~\ref{thm main}.

\begin{proof}[Proof of Theorem~\ref{thm main}]
Let $p \geq 3$ and $\phi_1,\dots,\phi_p \in \mathcal{C}^0(M)$. Let $d \in \N$ be large enough, let $s_d \in \R\H$ be a standard Gaussian vector and let $\nu_d$ denote the counting measure of the real zero set of $s_d$.

By Lemma~\ref{lem integral expression} and the fact that $M^\I = \bigsqcup_{\J \in \pa_\I} M^{\I,p}_{\J,d}$ for all $\I \in \pa_p$, we have:
\begin{equation*}
m_p(\nu_d)(\phi_1,\dots,\phi_p) = \sum_{\I \in \pa_p} \sum_{\J \in \pa_\I} \int_{\underline{x}_\I\in M^{\I,p}_{\J,d}} \iota_\I^* \phi(\underline{x}_\I) \D_d^\I(\underline{x}_\I) \rmes{M}^{\norm{\I}}.
\end{equation*}
By Lemma~\ref{lem order of the error terms}, up to an error term of the form $\left(\prod_{i=1}^p \Norm{\phi_i}_\infty \right)O\!\left(d^{\frac{1}{2}\floor*{\frac{p-1}{2}}}(\ln d)^p\right)$ we need only consider the terms in this double sum indexed by $(\I,\J) \in \mathfrak{C}_p$. The expression of these terms is given by Lemma~\ref{lem leading term}. Thus, we have:
\begin{multline}
\label{eq mp}
m_p(\nu_d)(\phi_1,\dots,\phi_p) = \sum_{(\I,\J) \in \mathfrak{C}_p} \left(\prod_{\{i,j\} \in \I \setminus \I_S} \int_M \phi_i\phi_j\mathcal{R}^1_d \rmes{M}\right)\left(\prod_{J \in \J_{\I_S}} \int_{M^J} \phi_J \D_d^J \rmes{M}^2 \right)\\
+ \left(\prod_{i=1}^p \Norm{\phi_i}_\infty\right) O\!\left(d^{\frac{1}{2}\floor*{\frac{p-1}{2}}}(\ln d)^p\right),
\end{multline}
where we used the same notations as in Definition~\ref{def CA}. In particular, recall that $S \subset \{1,\dots,p\}$ is the union of the singletons appearing in $\I$.

Let $(\I,\J) \in \mathfrak{C}_p$ and let us denote by $S \subset \{1,\dots,p\}$ the union of the singletons in $\I$. Since $\I_S = \{\{i\} \mid i \in S\}$, there is a canonical bijection $i \mapsto \{i\}$ from $S$ to $\I_S$. This first bijection induces a second one between $\pp_S$ and $\pp_{\I_S}$. Using this second bijection, $\J_{\I_S} \in \pp_{\I_S}$ corresponds to a partition into pairs of $S$, that we abusively denote by $\J_S$ in the following. Explicitly, we have $\J_S = \{\{i,j\} \mid \{\{i\},\{j\}\} \in \J_{\I_S}\} \in \pp_S$.

Let $J =\{\{i\},\{j\}\} \in \J_{\I_S}$ and let $I = \{i,j\}$ denote the corresponding pair in $\J_S$. Using Notation~\ref{ntn product indexed by A}, we have: $\phi_J = \phi_{\{i\}} \boxtimes \phi_{\{j\}} = \phi_i \boxtimes \phi_j = \phi_I$. Moreover, under the canonically identification $M^J \simeq M^I$, for all $(x,y) \in M^I$ we have:
\begin{equation*}
\D_d^J(x,y) = \mathcal{R}^2_d(x,y) - \mathcal{R}^1_d(x)\mathcal{R}^1_d(y) = \D_d^I(x,y).
\end{equation*}
Hence
\begin{equation*}
\int_{M^J} \phi_J \D_d^J \rmes{M}^2 = \int_{M^I} \phi_I \D_d^I \rmes{M}^2,
\end{equation*}
and the leading term on the right-hand side of Equation~\eqref{eq mp} equals:
\begin{equation}
\label{eq mp leading term}
\sum_{(\I,\J) \in \mathfrak{C}_p} \left(\prod_{\{i,j\} \in \I \setminus \I_S} \int_M \phi_i\phi_j\mathcal{R}^1_d \rmes{M}\right)\left(\prod_{I \in \J_S} \int_{M^I} \phi_I \D_d^I \rmes{M}^2 \right).
\end{equation}

We want to rewrite the leading term~\eqref{eq mp leading term} as a sum over $\{(\Pi,\Sigma) \mid \Pi \in \pp_p, \Sigma \subset \Pi\}$, by a change of variable. In order to do this, we now define a bijection from $\mathfrak{C}_p$ to $\{(\Pi,\Sigma) \mid \Pi \in \pp_p, \Sigma \subset \Pi\}$.
\begin{itemize}
\item Let $\Phi:\mathfrak{C}_p \to \{(\Pi,\Sigma) \mid \Pi \in \pp_p, \Sigma \subset \Pi\}$ be the map defined by:
\begin{equation*}
\Phi:(\I,\J) \mapsto \left((\I \setminus \I_S) \sqcup \J_S, \J_S\right).
\end{equation*}
For any $(\I,\J) \in \mathfrak{C}_p$, we have $\I \setminus \I_S \in \pp_{\{1,\dots,p\} \setminus S}$ and $\J_S \in \pp_S$. Hence, $(\I \setminus \I_S) \sqcup \J_S$ is indeed a partition of $\{1,\dots,p\}$ into pairs that contains $\J_S$.
\item We define a map $\Psi:\{(\Pi,\Sigma) \mid \Pi \in \pp_p, \Sigma \subset \Pi\} \to \mathfrak{C}_p$ as follows. Let $\Pi \in \pp_p$ and $\Sigma \subset \Pi$, we denote by $S = \sqcup_{I \in \Sigma} I$. We define a partition $\I = (\Pi \setminus \Sigma) \sqcup \{\{i\} \mid i \in S\} \in \pa_p$, so that $\I_S = \{\{i\} \mid i \in S\}$ and $\I_{\{1,\dots,p\} \setminus S} = \I \setminus \I_S = \Pi \setminus \Sigma$ is partition into pairs of $\{1,\dots,p\} \setminus S$. Then we define $\J \in \pa_\I$ by:
\begin{equation*}
\J = \left\{\rule{0em}{2.5ex} \{I\} \mvert I \in \I \setminus \I_S\right\} \sqcup \left\{\rule{0em}{2.5ex} \{\{i\},\{j\}\} \mvert \{i,j\} \in \Sigma\right\}.
\end{equation*}
Then $\J_{\I \setminus \I_S} = \{\{I\} \mid I \in \I \setminus \I_S\} = \J \setminus \J_{\I_S}$ and $\J_{\I_S}$ is a partition into pairs of $\I_S$. Hence $(\I,\J) \in \mathfrak{C}_p$, and we set $\Psi(\Pi,\Sigma)=(\I,\J)$. Note that we have $\J_S = \Sigma$.
\end{itemize}
By construction, $\Phi$ is a bijection from $\mathfrak{C}_p$ to $\{(\Pi,\Sigma) \mid \Pi \in \pp_p, \Sigma \subset \Pi\}$ such that $\Psi = \Phi^{-1}$. Using this bijection to change variables, the sum written in Equation~\eqref{eq mp leading term} becomes:
\begin{equation}
\label{eq final mp}
\sum_{\Pi \in \pp_p} \sum_{\Sigma \subset \Pi} \left(\prod_{\{i,j\} \in \Pi \setminus \Sigma} \int_M \phi_i\phi_j\mathcal{R}^1_d \rmes{M}\right)\left(\prod_{I \in \Sigma} \int_{M^I} \phi_I \D_d^I \rmes{M}^2 \right).
\end{equation}

On the other hand, by Lemma~\ref{lem integral expression} applied with $p=2$, for any $i$ and $j \in \{1,\dots,p\}$ distinct we have:
\begin{equation*}
m_2(\nu_d)(\phi_i,\phi_j) =\int_{M^I} \phi_I \D_d^I \rmes{M}^2 + \int_M \phi_i\phi_j \mathcal{R}^1_d \rmes{M},
\end{equation*}
where we denoted $I = \{i,j\}$. Thus,
\begin{equation*}
\sum_{\Pi \in \pp_p} \prod_{\{i,j\} \in \Pi} m_2(\nu_d)(\phi_i,\phi_j)
\end{equation*}
is equal to the term in Equation~\eqref{eq final mp}, hence is the leading term in Equation~\eqref{eq mp}. This proves the first claim in Theorem~\ref{thm main}. We obtain the second expression of $m_p(\nu_d)(\phi_1,\dots,\phi_p)$ in Theorem~\ref{thm main} by replacing $m_2(\nu_d)$ by its asymptotics, computed in Theorem~\ref{thm variance}.

Let $\phi \in \mathcal{C}^0(M)$. Using what we just proved with $\phi_i=\phi$ for all $i \in \{1,\dots,p\}$, we have:
\begin{equation*}
m_p(\prsc{\nu_d}{\phi}) = m_p(\nu_d)(\phi,\dots,\phi) = \card(\pp_p) \left(m_2(\nu_d)(\phi,\phi)\right)^\frac{p}{2} + \Norm{\phi}_\infty^p O\!\left(d^{\frac{1}{2}\floor*{\frac{p-1}{2}}}(\ln d)^p\right).
\end{equation*}
By definition, $m_2(\nu_d)(\phi,\phi)=\var{\prsc{\nu_d}{\phi}}$. Besides, the cardinality of $\pp_p$ is equal to $0$ if $p$ is odd, and to $2^{-\frac{p}{2}}\left(\frac{p}{2}!\right)^{-1}p!$ if $p$ is even. In both cases we have $\card(\pp_p)=\mu_p$ (see Definition~\ref{def mu p}). This proves the first expression of $m_p(\prsc{\nu_d}{\phi})$ stated in Theorem~\ref{thm main}. We obtain the second expression of $m_p(\prsc{\nu_d}{\phi})$ by using the expression of $\var{\prsc{\nu_d}{\phi}}$ given by Theorem~\ref{thm variance}.
\end{proof}


\section{Proof of the corollaries of Theorem~\ref{thm main}}
\label{sec proof of the corollaries}

In this section, we prove the corollaries of Theorem~\ref{thm main}. The strong Law of Large Numbers (Theorem~\ref{thm law of large numbers}) is proved in Section~\ref{subsec proof of the LLN}. The Central Limit Theorem (Theorem~\ref{thm central limit theorem}) is proved in Section~\ref{subsec proof of the CLT}. Finally, we prove Corollaries~\ref{cor concentration} and~\ref{cor hole probability} in Section~\ref{subsec proof of the corollaries}.


\subsection{Proof of the strong Law of Large Numbers (Theorem~\ref{thm law of large numbers})}
\label{subsec proof of the LLN}

The purpose of this section is to prove the strong Law of Large Numbers (Theorem~\ref{thm law of large numbers}). This result follows from the moments estimates of Theorem~\ref{thm main} by a Borel--Cantelli type argument, using the separability of $\mathcal{C}^0(M)$. This method was already used in~\cite{Let2019,LP2019,SZ1999}, for example.

We follow the notation of Section~\ref{sec framework and background}. In particular, recall that $(s_d)_{d \geq 0}$ is a sequence of random vectors such, for all $d \geq 0$, $s_d \sim \mathcal{N}(0,\Id)$ in $\R\H$. Then, $Z_d$ is the real zero set of $s_d$ and $\nu_d$ is the counting measure of $Z_d$.

\begin{proof}[Proof of Theorem~\ref{thm law of large numbers}]
Let us first consider the case of one test-function $\phi \in \mathcal{C}^0(M)$. We have:
\begin{equation*}
\esp{\sum_{d \geq 1} \left(\frac{\prsc{\nu_d}{\phi} - \esp{\prsc{\nu_d}{\phi}}}{d^\frac{1}{2}}\right)^6} = \sum_{d \geq 1} d^{-3} m_6(\nu_d)(\phi,\dots,\phi).
\end{equation*}
By Theorem~\ref{thm main} this sum is finite. Indeed, $m_6(\nu_d)(\phi,\dots,\phi) = O(d^\frac{3}{2})$. Then, almost surely, we have:
\begin{equation*}
\sum_{d \geq 1} \left(\frac{\prsc{\nu_d}{\phi} - \esp{\prsc{\nu_d}{\phi}}}{d^\frac{1}{2}}\right)^6 < +\infty,
\end{equation*}
hence $d^{-\frac{1}{2}}\left(\prsc{\nu_d}{\phi} - \rule{0pt}{10pt} \esp{\prsc{\nu_d}{\phi}}\right) \xrightarrow[d \to +\infty]{} 0$. Thus, for all $\phi \in \mathcal{C}^0(M)$, we have almost surely:
\begin{equation*}
d^{-\frac{1}{2}} \prsc{\nu_d}{\phi} = d^{-\frac{1}{2}} \esp{\prsc{\nu_d}{\phi}} + o(1) = \frac{1}{\pi} \int_M \phi \rmes{M} + o(1).
\end{equation*}
Applying this to $\phi = \mathbf{1}$, the constant unit function, we obtain: $d^{-\frac{1}{2}}\card(Z_d) \xrightarrow[d \to +\infty]{} \frac{1}{\pi}\vol{M}$ almost surely.

Now, recall that $\left(\mathcal{C}^0(M),\Norm{\cdot}_\infty\right)$ is separable. Let $(\phi_k)_{k \geq 0}$ denote a dense sequence in this space such that $\phi_0 =\mathbf{1}$. Almost surely, for all $k \geq 0$, we have $d^{-\frac{1}{2}} \prsc{\nu_d}{\phi_k} \xrightarrow[d \to +\infty]{} \frac{1}{\pi} \int_M \phi_k \rmes{M}$. Let $\underline{s}=(s_d)_{d \geq 1} \in \prod_{d \geq 1} \R\H$ denote a fixed sequence belonging to the probability $1$ event on which this condition holds. For every $\phi \in \mathcal{C}^0(M)$ and $k \in \N$, we have:
\begin{multline*}
\norm{d^{-\frac{1}{2}}\prsc{\nu_d}{\phi} - \frac{1}{\pi} \int_M \phi \rmes{M}}\\
\begin{aligned}
&\leq d^{-\frac{1}{2}}\norm{\prsc{\nu_d}{\phi-\phi_k}} + \norm{d^{-\frac{1}{2}}\prsc{\nu_d}{\phi_k} - \frac{1}{\pi} \int_M \phi_k \rmes{M}} + \frac{1}{\pi} \int_M \norm{\phi-\phi_k} \rmes{M}\\
&\leq \Norm{\phi-\phi_k}_\infty \left(d^{-\frac{1}{2}}\card(Z_d) + \frac{\vol{M}}{\pi}\right) + \norm{d^{-\frac{1}{2}}\prsc{\nu_d}{\phi_k} - \frac{1}{\pi} \int_M \phi_k \rmes{M}}.
\end{aligned}
\end{multline*}
Since $\phi_0=\mathbf{1}$, the sequence $(d^{-\frac{1}{2}}\card(Z_d))_{d \geq 1}$ converges, hence is bounded by some $K >0$. Let $\epsilon >0$, there exists $k \in \N$ such that $\Norm{\phi-\phi_k}_\infty < \epsilon$. Then, for every $d$ large enough, we have:
\begin{equation*}
\norm{d^{-\frac{1}{2}}\prsc{\nu_d}{\phi_k} - \frac{1}{\pi} \int_M \phi_k \rmes{M}} < \epsilon,
\end{equation*}
hence
\begin{equation*}
\norm{d^{-\frac{1}{2}}\prsc{\nu_d}{\phi} - \frac{1}{\pi} \int_M \phi \rmes{M}} < \epsilon \left(1 + K +\frac{\vol{M}}{\pi}\right).
\end{equation*}
Thus, for all $\phi \in \mathcal{C}^0(M)$, we have $d^{-\frac{1}{2}}\prsc{\nu_d}{\phi} \xrightarrow[d \to +\infty]{} \frac{1}{\pi} \int_M \phi \rmes{M}$, which concludes the proof.
\end{proof}


\subsection{Proof of the Central Limit Theorem (Theorem~\ref{thm central limit theorem})}
\label{subsec proof of the CLT}

In this section, we prove the Central Limit Theorem (Theorem~\ref{thm central limit theorem}). The result follows from Theorem~\ref{thm main} by the method of moments, see~\cite[Chap.~30]{Bil1995}. This method allows to prove the Central Limit Theorem for any fixed continuous test-function. Then, we apply the Lévy--Fernique Theorem~\cite[Theorem~III.6.5]{Fer1967}, which is the analogue of Lévy's Continuity Theorem for random generalized functions. Using this theorem, we prove the convergence in distribution to the Standard Gaussian White Noise in $\D'(M)$, the space of generalized functions on $M$.

\begin{proof}[Proof of Theorem~\ref{thm central limit theorem}]
\textbf{Central Limit Theorem for a fixed test-function.}
Let $\phi \in \mathcal{C}^0(M)$ be such that $\phi \neq 0$. We define a sequence $(X_d)$ of centered and normalized real random variables by:
\begin{equation*}
X_d = \frac{\prsc{\nu_d}{\phi} - \esp{\prsc{\nu_d}{\phi}}}{\var{\prsc{\nu_d}{\phi}}^\frac{1}{2}}.
\end{equation*}
Note that since $\phi \neq 0$, by Theorem~\ref{thm variance}, $\var{\prsc{\nu_d}{\phi}} = m_2(\nu_d)(\phi,\phi)$ is positive for $d$ large enough. Hence $X_d$ is well-defined for $d$ large enough, and we want to prove that $X_d \xrightarrow[d \to +\infty]{} \mathcal{N}(0,1)$ in distribution. By Theorem~\ref{thm variance} and~\ref{thm main}, for any integer $p \geq 3$, we have:
\begin{equation*}
\esp{X_d^p} = \frac{m_p(\nu_d)(\phi,\dots,\phi)}{\var{\prsc{\nu_d}{\phi}}^\frac{p}{2}} \xrightarrow[d \to +\infty]{} \mu_p,
\end{equation*}
where $\mu_p$ is the $p$-th moment of an $\mathcal{N}(0,1)$ variable (recall Definition~\ref{def mu p}). By the Theorem of Moments (cf.~\cite[Theorem~30.2]{Bil1995}), this implies that $X_d \xrightarrow[d \to +\infty]{}\mathcal{N}(0,1)$ in distribution. Replacing the expectation and the variance of $\prsc{\nu_d}{\phi}$ by their asymptotics (Theorem~\ref{thm expectation} and Theorem~\ref{thm variance}), we get:
\begin{equation*}
X_d = \frac{1}{d^\frac{1}{4}\sigma} \left(\int_M \phi^2 \rmes{M})\right)^{-\frac{1}{2}} \left(\prsc{\nu_d}{\phi}- \frac{d^\frac{1}{2}}{\pi} \int_M \phi \rmes{M}\right) +o(1).
\end{equation*}
Hence,
\begin{equation*}
\frac{1}{d^\frac{1}{4}\sigma} \left(\prsc{\nu_d}{\phi}- \frac{d^\frac{1}{2}}{\pi} \int_M \phi \rmes{M}\right) \xrightarrow[d \to +\infty]{} \mathcal{N}(0,\Norm{\phi}^2_2),
\end{equation*}
in distribution, where $\Norm{\phi}_2^2 = \int_M \phi^2 \rmes{M}$. Of course, this also holds for $\phi=0$. Taking $\phi = \mathbf{1}$ yields the Central Limit Theorem for the cardinality of $Z_d$.

\textbf{Central Limit Theorem in $\D'(M)$.}
Given $d \geq 1$, let us consider the random measure $T_d = \left(d^\frac{1}{4}\sigma\right)^{-1} \left(\nu_d - d^\frac{1}{2}\rmes{M}\right)$. This measure defines a random generalized function. Recall that its characteristic functional is the map $\chi_d :\phi \mapsto \esp{\exp\left(i\prsc{T_d}{\phi}_{(\D',\mathcal{C}^\infty)}\right)}$ from $\mathcal{C}^\infty(M)$ to $\C$. Let $\phi \in \mathcal{C}^\infty(M)$, we just proved that $\prsc{T_d}{\phi}_{(\D',\mathcal{C}^\infty)} = \prsc{T_d}{\phi} \xrightarrow[d \to +\infty]{} \mathcal{N}(0,\Norm{\phi}_2^2)$. Hence,
\begin{equation*}
\chi_d(\phi) = \esp{e^{i\prsc{T_d}{\phi}}} \xrightarrow[d \to +\infty]{} e^{-\frac{1}{2}\Norm{\phi}_2^2},
\end{equation*}
where we recognized the characteristic function of $\prsc{T_d}{\phi}$ evaluated at $1$. Thus, the sequence $(\chi_d)$ converges pointwise to $\chi:\phi \mapsto e^{-\frac{1}{2}\Norm{\phi}_2^2}$, which is the characteristic functional of the Standard Gaussian White Noise~$W$ (recall Definition~\ref{def W}). Then, by the Lévy--Fernique Theorem (cf.~\cite[Theorem~III.6.5]{Fer1967}), we have $T_d \xrightarrow[d \to +\infty]{} W$ in distribution, as random variables in $\D'(M)$.

\textbf{Central Limit Theorem for a family of smooth test-functions.}
Let $\phi_1,\dots,\phi_k \in \mathcal{C}^\infty(M)$, by the Continuous Mapping Theorem, we have:
\begin{equation*}
\left(\prsc{T_d}{\phi_1}_{(\D',\mathcal{C}^\infty)},\dots,\prsc{T_d}{\phi_k}_{(\D',\mathcal{C}^\infty)}\right) \xrightarrow[d \to +\infty]{} \left(\prsc{W}{\phi_1}_{(\D',\mathcal{C}^\infty)},\dots,\prsc{W}{\phi_k}_{(\D',\mathcal{C}^\infty)}\right)
\end{equation*}
in distribution, as random vectors in $\R^k$. This yields the joint Central Limit Theorem for a family of smooth test-functions.
\end{proof}


\subsection{Proof of Corollaries~\ref{cor concentration} and~\ref{cor hole probability}}
\label{subsec proof of the corollaries}

We first prove Corollary~\ref{cor concentration}, that is the concentration in probability. This is just an application of Markov's Inequality.

\begin{proof}[Proof of Corollary~\ref{cor concentration}]
Let $(\epsilon_d)_{d \geq 1}$ be a positive sequence and let $p \in \N^*$. By Markov's inequality for the $2p$-th moment, we have:
\begin{equation*}
\P\left(d^{-\frac{1}{2}}\norm{\prsc{\nu_d}{\phi} - \esp{\prsc{\nu_d}{\phi}}} > \epsilon_d \right) \leq \epsilon_d^{-2p} d^{-p} m_{2p}(\nu_d)(\phi,\dots,\phi).
\end{equation*}
By Theorem~\ref{thm main}, the term on the right-hand side is $O\!\left((d^\frac{1}{4}\epsilon_d)^{-2p}\right)$. As usual, one obtains the statement about $\card(Z_d)$ by specializing the result for $\phi = \mathbf{1}$.
\end{proof}

The proof of Corollary~\ref{cor hole probability}, that is of the Hole probability, uses Corollary~\ref{cor concentration} with the right test-function and a constant sequence $(\epsilon_d)_{d \geq 1}$. It is similar to the proof of~\cite[Corollary~1.10]{Let2019}.

\begin{proof}[Proof of Corollary~\ref{cor hole probability}]
Let $U$ be a non-empty open subset of $M$, and let $\phi_U:M \to \R$ be a continuous function such that for all $x \in M$, $\phi_U(x)>0$ if $x \in U$, and $\phi_U(x)=0$ otherwise. Let $\epsilon >0$ be such that $\epsilon < \frac{1}{\pi} \int_M \phi_U \rmes{M}$. By Theorem~\ref{thm expectation}, for all $d$ large enough we have $d^{-\frac{1}{2}}\esp{\prsc{Z_d}{\phi_U}} > \epsilon$. For a degree $d$ such that this condition holds, we have:
\begin{equation*}
\P\left(Z_d \cap U = \emptyset\right) = \P\left(\prsc{Z_d}{\phi_U} = 0\right) \leq \P\left(d^{-\frac{1}{2}}\norm{\prsc{Z_d}{\phi_U}-\rule{0pt}{10pt}\esp{\prsc{Z_d}{\phi_U}}}> \epsilon\right).
\end{equation*}
By Corollary~\ref{cor concentration}, the right-hand side is $O\!\left(d^{-\frac{p}{2}}\right)$ for all $p \in \N^*$.
\end{proof}


\bibliographystyle{amsplain}
\bibliography{RootsOfKostlanPolynomials}

\end{document}